\documentclass[final,1p]{elsarticle}

\usepackage{amssymb}
\usepackage{amsthm}

\usepackage{amsfonts,amsmath}

\usepackage[colorlinks, linkcolor=blue, citecolor=blue, pdfborder={0 0 0 [0 0]}]{hyperref}

\numberwithin{equation}{section}

\theoremstyle{plain}
\newtheorem{theorem}{Theorem}[section]
\newtheorem{proposition}{Proposition}[section]
\newtheorem{lemma}{Lemma}[section]

\theoremstyle{definition}
\newtheorem{definition}{Definition}[section]

\theoremstyle{remark}
\newtheorem{remark}{Remark}[section]

\journal{Journal of Differential Equations}

\usepackage{mathtools}

\def\cl{\operatorname{cl}}
\def\rd{\mathrm{d}}
\def\Lip{\operatorname{Lip}}

\begin{document}

\begin{frontmatter}

\title{Path-dependent Hamilton--Jacobi equations: \\
Uniqueness results for viscosity solutions \\ defined via families of compact sets\tnoteref{label1}}

\tnotetext[label1]{This work is supported by RSF grant 25-11-00269, \href{https://rscf.ru/en/project/25-11-00269/}{https://rscf.ru/en/project/25-11-00269/}.}

\author[aff1,aff2]{M.I. Gomoyunov}
\ead{m.i.gomoyunov@gmail.com}

\affiliation[aff1]{organization={N.N. Krasovskii Institute of Mathematics and Mechanics of the Ural Branch of the Russian Academy of Sciences},
            addressline={16 S. Kovalevskaya Str.},
            city={Ekaterinburg},
            postcode={620077},
            country={Russia}}

\affiliation[aff2]{organization={Ural Federal University},
            addressline={19 Mira Str.},
            city={Ekaterinburg},
            postcode={620062},
            country={Russia}}

\begin{abstract}
    We consider a path-dependent Hamilton--Jacobi equation with coinvariant derivatives over the space of continuous functions.
    We prove two uniqueness results for viscosity (generalized) solutions defined in terms of coinvariantly smooth test functionals and a dense family of compact subsets of the space of continuous functions.
    It is assumed that the Hamiltonian is continuous and satisfies a local Lipschitz condition in the functional variable with respect to the supremum norm.
    When the Lipschitz constant satisfies a sublinear growth condition in the gradient (impulse) variable, uniqueness is established in the class of continuous viscosity solutions.
    In the general case, without any such growth conditions, uniqueness is established in the class of continuous viscosity solutions that satisfy an additional local Lipschitz condition.
    The proofs are based on the standard method of doubling variables, but use a novel penalty functional for constructing coinvariantly smooth test functionals.
    The obtained results generalize previously known ones by relaxing the assumptions on the Hamiltonian and/or enlarging the class of functionals in which uniqueness is established.
\end{abstract}

\begin{keyword}
Path-dependent Hamilton--Jacobi equation \sep
Coinvariant derivatives \sep
Viscosity solution \sep
Uniqueness theorem \sep
Method of doubling variables

\MSC[2020] 35R15 \sep 35F21 \sep 35D40
\end{keyword}

\end{frontmatter}

\section{Introduction}

    In this paper, we consider issues related to the development the theory of viscosity (generalized) solutions of Hamilton--Jacobi equations with first-order partial derivatives (see, e.g., \cite{Crandall_Lions_1983,Crandall_Evans_Lions_1984}) for the case of path-dependent Hamilton--Jacobi equations with coinvariant derivatives over the space of continuous functions.
    Equations of this type arise in (deterministic) optimal control problems and differential games for time-delay systems.
    For further details, the reader is referred to the survey paper \cite{Gomoyunov_Lukoyanov_RMS_2024}.

    We define a viscosity solution in terms of coinvariantly smooth test functionals and a dense family of compact subsets of the space of continuous functions.
    This approach was originally proposed for Hamilton--Jacobi equations in infinite dimensions in \cite{Soner_1988}.
    For the considered case of path-dependent Hamilton--Jacobi equations, it was first adapted in \cite{Lukoyanov_2007_DE_Eng} and subsequently developed in \cite{Lukoyanov_2007_IMM_Eng,
    Lukoyanov_2010_IMM_Eng_1,
    Kaise_2015,
    Kaise_Kato_Takahashi_2018,
    Zhou_2020_1,
    Kaise_2022,
    Hernandez-Hernandez_Kaise_2024,
    Kaise_2025}.
    Note that there are other approaches of how to define a viscosity solution of a path-dependent Hamilton--Jacobi equation.
    We refer the reader to the discussion in \cite[Section 3.8]{Gomoyunov_Lukoyanov_RMS_2024}.

    The main motivation for this paper is to generalize the uniqueness result for viscosity solutions from \cite[Theorem 4]{Lukoyanov_2010_IMM_Eng_1} (see also \cite[Theorem 8]{Gomoyunov_Lukoyanov_RMS_2024}; particular cases were established earlier in \cite[Theorem 5.1]{Lukoyanov_2007_DE_Eng} and \cite[Theorem 2]{Lukoyanov_2007_IMM_Eng}) to the case where the Hamiltonian satisfies a local Lipschitz condition in the functional variable with respect to the supremum norm.
    Such Lipschitz conditions are quite general and natural in view of applications to optimal control problems and differential games for time-delay systems.
    Observe that uniqueness results for viscosity solutions under similar Lipschitz conditions were obtained in \cite[Theorem 5.2]{Kaise_2022} (particular cases were established earlier in \cite[Theorem 3.3]{Kaise_2015} and \cite[Theorem 5.6]{Kaise_Kato_Takahashi_2018}), in \cite[Theorem 5.1]{Zhou_2020_1}, and in \cite[Theorem 3.2]{Hernandez-Hernandez_Kaise_2024} (see also \cite[Theorem 2.4]{Kaise_2025}).
    However, in the cited results, some additional assumptions on the Hamiltonian are imposed, and uniqueness is proved in classes of viscosity solutions satisfying additional Lipschitz and/or growth conditions.

    We prove two uniqueness results for viscosity solutions of the considered path-depen\-dent Hamilton--Jacobi equation under the assumptions that the Hamiltonian is continuous and satisfies a local Lipschitz condition in the functional variable with respect to the supremum norm.
    In the first result (Theorem \ref{theorem_global_uniqueness_1}), we assume that the Lipschitz constant satisfies a sublinear growth condition in the gradient (impulse) variable;
    uniqueness is proved in the class of continuous functionals.
    The second result (Theorem \ref{theorem_global_uniqueness_2}) addresses a more general case without any such growth conditions; uniqueness is then proved in the narrower class of continuous functionals that satisfy an additional local Lipschitz condition.
    These two results generalize \cite[Theorem 4]{Lukoyanov_2010_IMM_Eng_1}, \cite[Theorem 5.2]{Kaise_2022}, \cite[Theorem 5.1]{Zhou_2020_1}, and \cite[Theorem 3.2]{Hernandez-Hernandez_Kaise_2024} by requiring weaker assumptions on the Hamiltonian and/or by establishing uniqueness in a wider class of functionals.

    The proofs use the method of doubling variables and are based on fairly standard arguments developed for Hamilton--Jacobi equations with first-order partial derivatives in, e.g., \cite[Theorem 3.7]{Bardi_Capuzzo-Dolcetta_1997} and subsequently adapted to path-dependent Hamilton--Jacobi equations with coinvariant derivatives in, e.g., \cite[Theorem 2]{Lukoyanov_2007_IMM_Eng}.
    The generalization of the previously known results is achieved through a novel penalty functional used in the construction of coinvariantly smooth test functionals.
    This penalty functional can be viewed as one of possible extensions of the functional proposed in \cite{Zhou_2020_1} (see also, e.g., \cite{Gomoyunov_Lukoyanov_Plaksin_2021,Zhou_2022}) to the case of doubled variables.
    Note that a straightforward extension to the case of doubled variables yields a functional that fails to be coinvariantly differentiable everywhere \cite[Example 4.1]{Gomoyunov_Plaksin_2023_JFA} and, therefore, cannot be directly taken as a penalty functional.
    Thus, the design of the new penalty functional constitutes the main feature of the presented proofs.

    The paper is organized as follows.
    In Section \ref{section_HJ}, the path-dependent Hamilton--Jacobi equation under study is described and the considered definitions of viscosity solutions are given.
    In Section \ref{section_main_results}, the main results of the paper are formulated and then compared with those known in the literature.
    In Section \ref{section_functional}, the penalty functional is defined and its properties are derived.
    In Section \ref{section_proofs}, the proofs of the main results are provided.

\section{Path-dependent Hamilton--Jacobi equation and viscosity solutions}
\label{section_HJ}

    Let numbers $n \in \mathbb{N}$, $T > 0$, and $h \geq 0$ be fixed.
    We denote by $C \coloneq C([- h, T], \mathbb{R}^n)$ the Banach space of all continuous functions $x \colon [- h, T] \to \mathbb{R}^n$ equipped with the supremum norm $\|x(\cdot)\|_\infty \coloneq \max_{\xi \in [- h, T]} \|x(\xi)\|$, where $\|\cdot\|$ stands for the Euclidean norm in $\mathbb{R}^n$.
    For any $a$, $b \in \mathbb{R}$, we put $a \wedge b \coloneq \min \{a, b\}$ and $a \vee b \coloneq \max \{a, b\}$.
    Accordingly, given a point $(t, x(\cdot)) \in [0, T] \times C$, we define a function $x(\cdot \wedge t) \in C$ by $x(\xi \wedge t) \coloneq x(\xi)$ for all $\xi \in [- h, t]$ and $x(\xi \wedge t) \coloneq x(t)$ for all $\xi \in (t, T]$.
    For every $R > 0$, we denote by $B(R)$ the closed ball in $\mathbb{R}^n$ centered at the origin of radius $R$.

    Given a non-empty set $D \subset C$ such that
    \begin{equation} \label{D_property_1}
        x(\cdot \wedge t)
        \in D
        \quad \forall (t, x(\cdot)) \in [0, T) \times D,
    \end{equation}
    a functional $\varphi \colon [0, T] \times D \to \mathbb{R}$ is called non-anticipative if
    \begin{equation*}
        \varphi(t, x(\cdot))
        = \varphi(t, x(\cdot \wedge t))
        \quad \forall (t, x(\cdot)) \in [0, T] \times D.
    \end{equation*}

    For every point $(t, x(\cdot)) \in [0, T] \times C$, let $\Lip(t, x(\cdot))$ be the set of all functions $z(\cdot) \in C$ such that $z(\cdot \wedge t) = x(\cdot \wedge t)$ and the restriction $z \vert_{[t, T]}(\cdot)$ of the function $z(\cdot)$ to the interval $[t, T]$ is Lipschitz continuous.
    According to \cite{Kim_1999,Lukoyanov_2007_IMM_Eng} (see also, e.g., \cite[Section 3.1]{Gomoyunov_Lukoyanov_RMS_2024}), a functional $\varphi \colon [0, T] \times C \to \mathbb{R}$ is called coinvariantly differentiable ($ci$-differentiable) at a point $(t, x(\cdot)) \in [0, T) \times C$ if there exist $\partial_t \varphi(t, x(\cdot)) \in \mathbb{R}$ and $\nabla_{x(\cdot)} \varphi(t, x(\cdot)) \in \mathbb{R}^n$ such that, for every function $z(\cdot) \in \Lip(t, x(\cdot))$,
    \begin{align*}
        \lim_{\delta \to 0^+} \frac{1}{\delta}
        \bigl( & \varphi(t + \delta, z(\cdot)) - \varphi(t, x(\cdot)) \\
        & \quad - \partial_t \varphi(t, x(\cdot)) \delta
        - \langle \nabla_{x(\cdot)} \varphi(t, x(\cdot)), z(t + \delta) - x(t) \rangle \bigr)
        = 0,
    \end{align*}
    where $\langle \cdot, \cdot \rangle$ stands for the inner product in $\mathbb{R}^n$.
    In this case, $\partial_t \varphi(t, x(\cdot))$ and $\nabla_{x(\cdot)} \varphi(t, x(\cdot))$ are called coinvariant derivatives ($ci$-derivatives) of $\varphi$ at $(t, x(\cdot))$.
    Moreover, $\varphi$ is called coinvariantly smooth ($ci$-smooth) if it is continuous, $ci$-differentiable at every point $(t, x(\cdot)) \in [0, T) \times C$, and the mappings $\partial_t \varphi \colon [0, T) \times C \to \mathbb{R}$, $\nabla_{x(\cdot)} \varphi \colon [0, T) \times C \to \mathbb{R}^n$ are continuous.

    The paper deals with a path-dependent Hamilton--Jacobi equation with $ci$-derivatives of the following form:
    \begin{equation} \label{HJ}
        \partial_t \varphi(t, x(\cdot))
        + H \bigl( t, x(\cdot), \nabla_{x(\cdot)} \varphi(t, x(\cdot)) \bigr)
        = 0.
    \end{equation}
    Here, $(t, x(\cdot)) \in [0, T) \times C$, the Hamiltonian $H \colon [0, T] \times C \times \mathbb{R}^n \to \mathbb{R}$ is given, and the non-anticipative functional $\varphi \colon [0, T] \times C \to \mathbb{R}$ is unknown.

    In accordance with, e.g., \cite{Lukoyanov_2007_IMM_Eng} (see also \cite[Section 3.8]{Gomoyunov_Lukoyanov_RMS_2024}), we give the following two definitions of viscosity subsolutions, supersolutions, and solutions of \eqref{HJ}.
    The first one is local in nature (in the sense that it deals with functionals $\varphi$ defined on $[0, T] \times D$ with a compact set $D \subset C$), and the second one is global.

    \begin{definition} \label{definition_viscosity_D}
        Let $D \subset C$ be a non-empty compact set that satisfies condition \eqref{D_property_1} and let $\varphi \colon [0, T] \times D \to \mathbb{R}$ be a non-anticipative continuous functional.
        \begin{itemize}
            \item[(i)]
            We call $\varphi$ a viscosity $D$-subsolution of \eqref{HJ} if, for every $ci$-smooth (test) func\-tional $\psi \colon [0, T] \times C \to \mathbb{R}$ and every point $(t, x(\cdot)) \in [0, T) \times D$ such that
            \begin{equation} \label{subsolution_condition}
                \varphi(t, x(\cdot)) - \psi(t, x(\cdot))
                = \max_{(\tau, y(\cdot)) \in [0, T] \times D}
                \bigl( \varphi(\tau, y(\cdot)) - \psi(\tau, y(\cdot)) \bigr),
            \end{equation}
            the inequality below is valid:
            \begin{equation*}
                \partial_t \psi(t, x(\cdot))
                + H\bigl(t, x(\cdot), \nabla_{x(\cdot)} \psi(t, x(\cdot)) \bigr)
                \geq 0.
            \end{equation*}

            \item[(ii)]
            We call $\varphi$ a viscosity $D$-supersolution of \eqref{HJ} if, for every $ci$-smooth (test) functional $\psi \colon [0, T] \times C \to \mathbb{R}$ and every point $(t, x(\cdot)) \in [0, T) \times D$ such that
            \begin{equation} \label{supersolution_condition}
                \varphi(t, x(\cdot)) - \psi(t, x(\cdot))
                = \min_{(\tau, y(\cdot)) \in [0, T] \times D}
                \bigl( \varphi(\tau, y(\cdot)) - \psi(\tau, y(\cdot)) \bigr),
            \end{equation}
            the inequality below is valid:
            \begin{equation*}
                \partial_t \psi(t, x(\cdot))
                + H\bigl(t, x(\cdot), \nabla_{x(\cdot)} \psi(t, x(\cdot)) \bigr)
                \leq 0.
            \end{equation*}

            \item[(iii)]
            We call $\varphi$ a viscosity $D$-solution of \eqref{HJ} if $\varphi$ is a viscosity $D$-subsolution and a viscosity $D$-supersolution of \eqref{HJ}.
        \end{itemize}
    \end{definition}

    \begin{definition} \label{definition_viscosity_D_p}
        Let $\varphi \colon [0, T] \times C \to \mathbb{R}$ be a non-anticipative continuous functional.
        Let $P$ be a non-empty set and let $\{D_p\}_{p \in P}$ be a family of non-empty compact subsets of $C$ such that
        \begin{equation} \label{D_p_property_1}
            x(\cdot \wedge t)
            \in D_p
            \quad \forall (t, x(\cdot)) \in [0, T) \times D_p, \ p \in P
        \end{equation}
        and
        \begin{equation} \label{D_p_property_2}
            \cl \bigcup_{p \in P} D_p
            = C,
        \end{equation}
        where $\cl$ denotes the closure in the space $C$.
        We call $\varphi$ a viscosity $\{D_p\}_{p \in P}$-subsolution (resp., $\{D_p\}_{p \in P}$-supersolution, $\{D_p\}_{p \in P}$-solution) of \eqref{HJ} if, for every $p \in P$, the restriction $\varphi \vert_{[0, T] \times D_p}$ of the functional $\varphi$ to the set $[0, T] \times D_p$ is a viscosity $D_p$-subsolution (resp., $D_p$-supersolution, $D_p$-solution) of \eqref{HJ}.
    \end{definition}

\section{Main results}
\label{section_main_results}

    In this section, we present the main results of the paper.
    First, we establish local results concerning viscosity $D$-solutions of the path-dependent Hamilton--Jacobi equation \eqref{HJ}.
    Then, we turn to global results for viscosity $\{D\}_{p \in P}$-solutions of \eqref{HJ}.
    Finally, we compare the presented results with those obtained previously.

\subsection{Local results}

    The following two comparison results hold.
    \begin{theorem} \label{theorem_local_comparison_1}
        Let $D \subset C$ be a non-empty compact set satisfying \eqref{D_property_1} and such that
        \begin{equation} \label{D_property_L}
            L_D
            \coloneq \sup\biggl\{ \frac{\|x(\xi_1) - x(\xi_2)\|}{|\xi_1 - \xi_2|}
            \colon x(\cdot) \in D, \ \xi_1, \xi_2 \in [0, T], \ \xi_1 \neq \xi_2 \biggr\}
            < + \infty.
        \end{equation}
        Suppose that the Hamiltonian $H$ satisfies the following assumptions.
        \begin{itemize}
            \item[$(A.1)$]
            The restriction $H \vert_{[0, T] \times D \times \mathbb{R}^n}$ of the mapping $H$ to the set $[0, T] \times D \times \mathbb{R}^n$ is continuous.

            \item[$(A.2)$]
            There exists a number $L_{H, D} > 0$ such that
            \begin{equation*}
                |H(t, x(\cdot), s) - H(t, y(\cdot), s)|
                \leq L_{H, D} (1 + \|s\|) \|x(\cdot \wedge t) - y(\cdot \wedge t)\|_\infty
            \end{equation*}
            for all $t \in [0, T]$, $x(\cdot)$, $y(\cdot) \in D$, $s \in \mathbb{R}^n$.
        \end{itemize}
        Then, for any viscosity $D$-subsolution $\varphi_1$ and $D$-supersolution $\varphi_2$ of \eqref{HJ} such that
        \begin{equation} \label{varphi_1_varphi_2_boundary}
            \varphi_1(T, x(\cdot))
            \leq \varphi_2(T, x(\cdot))
            \quad \forall x(\cdot) \in D,
        \end{equation}
        the inequality below is valid:
        \begin{equation} \label{varphi_1_varphi_2_comparison}
            \varphi_1(t, x(\cdot))
            \leq \varphi_2(t, x(\cdot))
            \quad \forall (t, x(\cdot)) \in [0, T] \times D.
        \end{equation}
    \end{theorem}

    Note that, under assumption $(A.2)$, the functional $[0, T] \times D \ni (t, x(\cdot)) \mapsto H(t, x(\cdot), s)$ is non-anticipative for every $s \in \mathbb{R}^n$.

    The proof of Theorem \ref{theorem_local_comparison_1} is given in Section \ref{section_proofs}.

    For every $(t, x(\cdot)) \in [0, T) \times C$ and every $v \in B(1)$, let $z_{t, x(\cdot), v}(\cdot)$ be the (unique) function from $\Lip(t, x(\cdot))$ such that
    \begin{equation} \label{z_v}
        z_{t, x(\cdot), v}(\xi)
        = x(t) + v (\xi - t)
        \quad \forall \xi \in (t, T].
    \end{equation}
    Given a non-empty set $D \subset C$ such that
    \begin{equation} \label{D_property_v}
        z_{t, x(\cdot), v}(\cdot)
        \in D
        \quad \forall (t, x(\cdot)) \in [0, T) \times D, \ v \in B(1)
    \end{equation}
    and a functional $\varphi \colon [0, T] \times D \to \mathbb{R}$, denote
    \begin{equation} \label{L_varphi_definition}
        \begin{aligned}
            L_{\varphi, D}
            \coloneq \sup \biggl\{ & \frac{|\varphi(t + \delta, z_{t, x(\cdot), v}(\cdot))
            - \varphi(t, x(\cdot))|}{\delta}
            \colon \\
            & \qquad (t, x(\cdot)) \in [0, T) \times D, \ v \in B(1), \ \delta \in (0, T - t] \biggr\}.
        \end{aligned}
    \end{equation}
    Observe that \eqref{D_property_v} implies \eqref{D_property_1}, since $z_{t, x(\cdot), 0}(\cdot) = x(\cdot \wedge t)$.

    \begin{theorem} \label{theorem_local_comparison_2}
        Let $D \subset C$ be a non-empty compact set satisfying \eqref{D_property_L} and \eqref{D_property_v}.
        Suppose that the Hamiltonian $H$ satisfies $(A.1)$ and the following assumption.
        \begin{itemize}
            \item[$(A.3)$]
            For every $R > 0$, there exists a number $L_{H, D, R} > 0$ such that
            \begin{equation*}
                |H(t, x(\cdot), s) - H(t, y(\cdot), s)|
                \leq L_{H, D, R} \|x(\cdot \wedge t) - y(\cdot \wedge t)\|_\infty
            \end{equation*}
            for all $t \in [0, T]$, $x(\cdot)$, $y(\cdot) \in D$, $s \in B(R)$.
        \end{itemize}
        Let $\varphi_1$ and $\varphi_2$ be a viscosity $D$-subsolution and a viscosity $D$-supersolution of \eqref{HJ} respectively and let inequality \eqref{varphi_1_varphi_2_boundary} be fulfilled.
        In addition, suppose that
        \begin{equation} \label{varphi_1_varphi_2_Lipschitz}
            L_{\varphi_1, D} \vee L_{\varphi_2, D}
            < + \infty.
        \end{equation}
        Then, inequality \eqref{varphi_1_varphi_2_comparison} is valid.
    \end{theorem}

    Compared to Theorem \ref{theorem_local_comparison_1}, Theorem \ref{theorem_local_comparison_2} relaxes assumption $(A.2)$ to $(A.3)$, but requires in turn that a viscosity $D$-subsolution $\varphi_1$ and a viscosity $D$-superso\-lu\-tion $\varphi_2$ satisfy the additional condition \eqref{varphi_1_varphi_2_Lipschitz}, which can be regarded as a special Lipschitz condition for $\varphi_1$ and $\varphi_2$.

    The proof of Theorem \ref{theorem_local_comparison_2} is given in Section \ref{section_proofs}.
    It almost completely repeats that of Theorem \ref{theorem_local_comparison_1} except for the final step, where assumption $(A.2)$ is used.
    Roughly speaking, the relaxing of assumption $(A.2)$ to $(A.3)$ is possible because the Lipschitz condition \eqref{varphi_1_varphi_2_Lipschitz} allows us to derive suitable estimates for the $ci$-derivatives of the $ci$-smooth test functionals (see Lemma \ref{lemma} in Section \ref{section_proofs}).

    In view of the definition of a viscosity $D$-solution of \eqref{HJ} (see Definition \ref{definition_viscosity_D}, (iii)), Theorems \ref{theorem_local_comparison_1} and \ref{theorem_local_comparison_2} yield the two uniqueness results below.
    \begin{theorem} \label{theorem_local_uniqueness_1}
        Let $D \subset C$ be a non-empty compact set satisfying \eqref{D_property_1} and \eqref{D_property_L} and let the Hamiltonian $H$ satisfy assumptions $(A.1)$ and $(A.2)$.
        Then, there exists at most one viscosity $D$-solution $\varphi$ of \eqref{HJ}.
    \end{theorem}
    \begin{theorem} \label{theorem_local_uniqueness_2}
        Let $D \subset C$ be a non-empty compact set satisfying \eqref{D_property_L} and \eqref{D_property_v} and let the Hamiltonian $H$ satisfy assumptions $(A.1)$ and $(A.3)$.
        Then, there exists at most one viscosity $D$-solution $\varphi$ of \eqref{HJ} such that $L_{\varphi, D} < + \infty$.
    \end{theorem}

\subsection{Global results}

    According to Definition \ref{definition_viscosity_D_p}, the family $\{D_p\}_{p \in P}$ satisfies the density condition \eqref{D_p_property_2}, and viscosity $\{D_p\}_{p \in P}$-subsolutions and $\{D_p\}_{p \in P}$-super\-so\-lu\-tions of \eqref{HJ} are continuous.
    Hence, Theorems \ref{theorem_local_comparison_1} and \ref{theorem_local_comparison_2} immediately imply the following two comparison results.
    \begin{theorem} \label{theorem_global_comparison_1}
        Let $P$ be a non-empty set and let $\{D_p\}_{p \in P}$ be a family of non-empty compact subsets of $C$ satisfying \eqref{D_p_property_1} and \eqref{D_p_property_2} and such that
        \begin{equation} \label{D_p_property_L}
            \sup\biggl\{ \frac{\|x(\xi_1) - x(\xi_2)\|}{|\xi_1 - \xi_2|}
            \colon x(\cdot) \in D_p, \ \xi_1, \xi_2 \in [0, T], \ \xi_1 \neq \xi_2 \biggr\}
            < + \infty
            \quad \forall p \in P.
        \end{equation}
        Suppose that the Hamiltonian $H$ satisfies the following assumptions.
        \begin{itemize}
            \item[$(B.1)$]
            For every $p \in P$, the restriction $H \vert_{[0, T] \times D_p \times \mathbb{R}^n}$ is continuous.

            \item[$(B.2)$]
            For every $p \in P$, there exists a number $L_{H, D_p} > 0$ such that
            \begin{equation*}
                |H(t, x(\cdot), s) - H(t, y(\cdot), s)|
                \leq L_{H, D_p} (1 + \|s\|) \|x(\cdot \wedge t) - y(\cdot \wedge t)\|_\infty
            \end{equation*}
            for all $t \in [0, T]$, $x(\cdot)$, $y(\cdot) \in D_p$, $s \in \mathbb{R}^n$.
        \end{itemize}
        Then, for any viscosity $\{D_p\}_{p \in P}$-subsolution $\varphi_1$ and $\{D_p\}_{p \in P}$-supersolution $\varphi_2$ of \eqref{HJ} such that
        \begin{equation} \label{varphi_1_varphi_2_boundary_global}
            \varphi_1(T, x(\cdot))
            \leq \varphi_2(T, x(\cdot))
            \quad \forall x(\cdot) \in \bigcup_{p \in P} D_p,
        \end{equation}
        the inequality below is valid:
        \begin{equation} \label{varphi_1_varphi_2_comparison_global}
            \varphi_1(t, x(\cdot))
            \leq \varphi_2(t, x(\cdot))
            \quad \forall (t, x(\cdot)) \in [0, T] \times C.
        \end{equation}
    \end{theorem}
    \begin{theorem} \label{theorem_global_comparison_2}
        Let $P$ be a non-empty set and let $\{D_p\}_{p \in P}$ be a family of non-empty compact subsets of $C$ satisfying \eqref{D_p_property_2} and \eqref{D_p_property_L} and such that
        \begin{equation} \label{D_property_v_p}
            z_{t, x(\cdot), v}(\cdot)
            \in D_p
            \quad \forall (t, x(\cdot)) \in [0, T) \times D_p, \ v \in B(1), \ p \in P.
        \end{equation}
        Suppose that the Hamiltonian $H$ satisfies $(B.1)$ and the following assumption.
        \begin{itemize}
            \item[$(B.3)$]
            For every $p \in P$ and every $R > 0$, there exists a number $L_{H, D_p, R} > 0$ such that
            \begin{equation*}
                |H(t, x(\cdot), s) - H(t, y(\cdot), s)|
                \leq L_{H, D_p, R} \|x(\cdot \wedge t) - y(\cdot \wedge t)\|_\infty
            \end{equation*}
            for all $t \in [0, T]$, $x(\cdot)$, $y(\cdot) \in D_p$, $s \in B(R)$.
        \end{itemize}
        Let $\varphi_1$ and $\varphi_2$ be a viscosity $\{D_p\}_{p \in P}$-subsolution and a viscosity $\{D_p\}_{p \in P}$-supersolution  of \eqref{HJ} respectively and let inequality \eqref{varphi_1_varphi_2_boundary_global} be fulfilled.
        In addition, suppose that
        \begin{equation*}
            L_{\varphi_1, D_p} \vee L_{\varphi_2, D_p}
            < + \infty
            \quad \forall p \in P.
        \end{equation*}
        Then, inequality \eqref{varphi_1_varphi_2_comparison_global} is valid.
    \end{theorem}

    In a similar way, we conclude that the two uniqueness results below hold.
    \begin{theorem} \label{theorem_global_uniqueness_1}
        Let $P$ be a non-empty set, let $\{D_p\}_{p \in P}$ be a family of non-empty compact subsets of $C$ satisfying \eqref{D_p_property_1}, \eqref{D_p_property_2}, and \eqref{D_p_property_L}, and let the Hamiltonian $H$ satisfy assumptions $(B.1)$ and $(B.2)$.
        Then, there exists at most one viscosity $\{D\}_{p \in P}$-solution $\varphi$ of \eqref{HJ}.
    \end{theorem}
    \begin{theorem} \label{theorem_global_uniqueness_2}
        Let $P$ be a non-empty set, let $\{D_p\}_{p \in P}$ be a family of non-empty compact subsets of $C$ satisfying \eqref{D_p_property_2}, \eqref{D_p_property_L}, and \eqref{D_property_v_p}, and let the Hamiltonian $H$ satisfy assumptions $(B.1)$ and $(B.3)$.
        Then, there exists at most one viscosity $\{D_p\}_{p \in P}$-solution $\varphi$ of \eqref{HJ} such that
        \begin{equation} \label{varphi_local_L}
            L_{\varphi, D_p}
            < + \infty
            \quad \forall p \in P.
        \end{equation}
    \end{theorem}

    Thus, Theorem \ref{theorem_global_uniqueness_1} guarantees uniqueness of a viscosity solution of \eqref{HJ} in the class of continuous functionals $\varphi$ under assumptions $(B.1)$ and $(B.2)$ on $H$.
    Theorem \ref{theorem_global_uniqueness_2}, on the other hand, establishes uniqueness under the weaker assumptions $(B.1)$ and $(B.3)$ on $H$, but in the narrower class of continuous functionals $\varphi$ satisfying the additional local Lipschitz condition \eqref{varphi_local_L}.

\subsection{Remarks}

    Theorems \ref{theorem_global_uniqueness_1} and \ref{theorem_global_uniqueness_2} generalize the results obtained earlier on (global) uniqueness of viscosity solutions of path-dependent Hamilton--Jacobi equations of form \eqref{HJ} in the case where a viscosity solution is defined via one or another family $\{D_p\}_{p \in P}$ of compact subsets of the space $C$ in the spirit of Definitions \ref{definition_viscosity_D} and \ref{definition_viscosity_D_p}.
    Remarks \ref{remark_1}--\ref{remark_4} below discuss this issue in a detailed way (observe that we use the notation adopted in the present paper, which can differ from the notation used in the cited results).

    \begin{remark} \label{remark_1}
        Theorem \ref{theorem_global_uniqueness_1} generalizes \cite[Theorem 4]{Lukoyanov_2010_IMM_Eng_1} (see also \cite[Theorem 8]{Gomoyunov_Lukoyanov_RMS_2024}; particular cases were established earlier in \cite[Theorem 5.1]{Lukoyanov_2007_DE_Eng} and \cite[Theorem 2]{Lukoyanov_2007_IMM_Eng}).
        Namely, in \cite[Theorem 4]{Lukoyanov_2010_IMM_Eng_1}, the Hamiltonian $H$ was assumed to satisfy the following conditions.
        \begin{itemize}
            \item[$(C.1)$]
            The mapping $H$ is continuous.

            \item[$(C.2)$]
            There exists a number $c_H > 1$ such that
            \begin{align*}
                |H(t, x(\cdot), 0)|
                & \leq c_H (1 + \|x(\cdot \wedge t)\|_\infty), \\
                |H(t, x(\cdot), s) - H(t, x(\cdot), r)|
                & \leq c_H (1 + \|x(\cdot \wedge t)\|_\infty) \|s - r\|
            \end{align*}
            for all $t \in [0, T]$, $x(\cdot) \in C$, $s$, $r \in \mathbb{R}^n$.

            \item[$(C.3)$]
            There exist numbers $J \in \mathbb{N}$ and $\{\vartheta_j\}_{j \in \overline{1, J}} \subset (0, h]$ such that, for every compact set $D \subset C$, there exists a number $L_{H, D} > 1$ such that
            \begin{align*}
                & |H(t, x(\cdot), s) - H(t, y(\cdot), s)|
                \leq L_{H, D} (1 + \|s\|) \\
                & \ \ \times \left( \|x(t) - y(t)\|
                + \sum_{j = 1}^{J} \|x(t - \vartheta_j) - y(t - \vartheta_j)\|
                + \sqrt{\int_{- h}^{t} \|x(\xi) - y(\xi)\|^2 \, \rd \xi} \right)
            \end{align*}
            for all $t \in [0, T]$, $x(\cdot)$, $y(\cdot) \in D$, $s \in \mathbb{R}^n$.
        \end{itemize}
        Viscosity solutions were defined in \cite{Lukoyanov_2010_IMM_Eng_1} in terms of a sequence $\{D_k\}_{k \in \mathbb{N}}$ such that, for every $k \in \mathbb{N}$, the set $D_k$ consists of all Lipschitz continuous functions $x \colon [- h, T] \to \mathbb{R}^n$ such that $\|x(- h)\| \leq k$ and $\|\dot{x}(\xi)\| \leq k c_H (1 + \|x(\cdot \wedge \xi)\|_\infty)$ for a.e. $\xi \in [- h, T]$.
        Note that this sequence $\{D_k\}_{k \in \mathbb{N}}$ satisfies conditions \eqref{D_p_property_2}, \eqref{D_p_property_L}, and \eqref{D_property_v_p}.
        Thus, the advantages of Theorem \ref{theorem_global_uniqueness_1} over \cite[Theorem 4]{Lukoyanov_2010_IMM_Eng_1} are the following:
        first, assumption $(B.2)$ on local Lipschitz continuity of the Hamiltonian $H$ in the functional variable $x(\cdot)$ is expressed in terms of the supremum norm $\|\cdot\|_\infty$, which is weaker than $(C.3)$ and allows to cover more general classes of time-delay systems in applications;
        second, no assumptions on Lipschitz continuity of $H$ in the gradient variable $s$ such as $(C.2)$ are required.
    \end{remark}

    \begin{remark} \label{remark_2}
        Theorem \ref{theorem_global_uniqueness_1} also generalizes \cite[Theorem 3.2]{Hernandez-Hernandez_Kaise_2024} (see also \cite[Theorem 2.4]{Kaise_2025}).
        In that paper, viscosity solutions were defined via a sequence of compact sets $\{D_k\}_{k \in \mathbb{N}}$ constructed in the same way as in Remark \ref{remark_1} but with the number $c_H$ replaced by a given number $\alpha \geq 1$.
        The Hamiltonian $H$ was assumed to satisfy the condition below.
        \begin{itemize}
            \item[$(D)$]
                There exist numbers $L_H > 0$ and $M_H > 0$ and a family of Lipschitz moduli of continuity $\{\omega^{t, x(\cdot)}\}_{(t, x(\cdot)) \in [0, T] \times C}$ (see \cite[Definition 2.2]{Hernandez-Hernandez_Kaise_2024}) such that
                \begin{align*}
                    & |H(t, x(\cdot), s) - H(\tau, y(\cdot), r)|
                    \leq L_H (1 + \|s\|) \\
                    & \quad \times \bigl( (1 + \|x(\cdot \wedge t)\|_\infty) \omega^{t, x(\cdot)}(\tau - t) + \|x(\cdot \wedge t) - y(\cdot \wedge \tau)\|_\infty \bigr)
                    + M_H \|s - r\|
                \end{align*}
                for all $t$, $\tau \in [0, T]$ with $\tau \geq t$ and all $x(\cdot)$, $y(\cdot) \in C$, $s$, $r \in \mathbb{R}^n$.
        \end{itemize}
        Uniqueness was proved in the class of functionals $\varphi$ satisfying the following global Lipschitz condition:
        there exist numbers $L_\varphi > 0$ and $M_\varphi > 0$ such that
        \begin{align*}
            |\varphi(t, x(\cdot)) - \varphi(t, y(\cdot))|
            & \leq L_\varphi \|x(\cdot \wedge t) - y(\cdot \wedge t)\|_\infty, \\
            |\varphi(t, x(\cdot)) - \varphi(\tau, x(\cdot \wedge t))|
            & \leq M_\varphi (\tau - t)
        \end{align*}
        for all $t$, $\tau \in [0, T]$ with $\tau \geq t$ and all $x(\cdot)$, $y(\cdot) \in C$.
        Thus, Theorem \ref{theorem_global_uniqueness_1} improves upon \cite[Theorem 3.2]{Hernandez-Hernandez_Kaise_2024} in two respects:
        first, it imposes weaker assumptions on the Hamiltonian $H$;
        second, it establishes uniqueness in the wider class of merely continuous functionals $\varphi$, without requiring any Lipschitz continuity properties of $\varphi$.
        In addition, note that the definition of a viscosity solution adopted in \cite{Hernandez-Hernandez_Kaise_2024} slightly differs from Definitions \ref{definition_viscosity_D} and \ref{definition_viscosity_D_p}.
        Namely, it has an extra requirement that, for a test functional $\psi$ and a point $(t, x(\cdot))$ from conditions \eqref{subsolution_condition} and \eqref{supersolution_condition}, the inequality $\|\nabla_{x(\cdot)} \psi(t, x(\cdot))\| \leq L_\varphi + 2 M_\varphi$ should be fulfilled.
        However, the proof of Theorem \ref{theorem_local_comparison_2} shows that, in the case of functionals $\varphi$ satisfying the local Lipschitz condition \eqref{varphi_local_L}, similar a priory estimates can also be included into the definitions (in this connection, see also Lemma \ref{lemma}).
    \end{remark}

    \begin{remark} \label{remark_3}
        Theorem \ref{theorem_global_uniqueness_2} generalizes \cite[Theorem 5.2]{Kaise_2022} (particular cases were established earlier in \cite[Theorem 3.3]{Kaise_2015} and \cite[Theorem 5.6]{Kaise_Kato_Takahashi_2018}).
        In \cite{Kaise_2022}, viscosity solutions were defined as in Remark \ref{remark_2}, and it was assumed that the Hamiltonian $H$ satisfies the condition below.
        \begin{itemize}
            \item[$(E)$]
                There exists a family of Lipschitz moduli of continuity  $\{\omega^{t, x(\cdot)}\}_{(t, x(\cdot)) \in [0, T] \times C}$ such that, for every compact set $D \subset C$ and every number $R > 0$, there exists a number $L_{H, D, R} > 0$ such that
                \begin{equation} \label{D}
                    \begin{aligned}
                        |H(t, x(\cdot), s) - H(t, y(\cdot), s)|
                        & \leq L_{H, D, R} \|x(\cdot \wedge t) - y(\cdot \wedge t)\|_\infty, \\
                        |H(t, x(\cdot), s) - H(\tau, x(\cdot \wedge t), s)|
                        & \leq L_{H, D, R} \, \omega^{t, x(\cdot)}(\tau - t), \\
                        |H(t, x(\cdot), s) - H(t, x(\cdot), r)|
                        & \leq L_{H, D, R} \|s - r\|
                    \end{aligned}
                \end{equation}
                for all $t$, $\tau \in [0, T]$ with $\tau \geq t$ and all $x(\cdot)$, $y(\cdot) \in D$, $s$, $r \in B(R)$.
        \end{itemize}
        Uniqueness was proved in the class of functionals $\varphi$ satisfying the following global Lipschitz condition:
        there exist numbers $L_\varphi > 0$ and $M_\varphi > 0$ such that
        \begin{align*}
            |\varphi(t, x(\cdot)) - \varphi(t, y(\cdot))|
            & \leq L_\varphi (1 + \|x(\cdot \wedge t)\|_\infty + \|y(\cdot \wedge t)\|_\infty) \|x(\cdot \wedge t) - y(\cdot \wedge t)\|_\infty, \\
            |\varphi(t, x(\cdot)) - \varphi(\tau, x(\cdot \wedge t))|
            & \leq M_\varphi (1 + \|x(\cdot \wedge t)\|_\infty^2) (\tau - t)
        \end{align*}
        for all $t$, $\tau \in [0, T]$ with $\tau \geq t$ and all $x(\cdot)$, $y(\cdot) \in C$.
        So, the advantages of Theorem \ref{theorem_global_uniqueness_2} over \cite[Theorem 5.2]{Kaise_2022} are the following:
        first, no assumptions like Lipschitz continuity in the time variable $t$ and in the gradient variable $s$ such as the second and third conditions in \eqref{D} are required;
        second, uniqueness is established in the wider class of continuous functionals $\varphi$ satisfying the local Lipschitz condition \eqref{varphi_local_L}.
    \end{remark}

    \begin{remark} \label{remark_4}
        A result close to Theorems \ref{theorem_global_uniqueness_1} and \ref{theorem_global_uniqueness_2} was announced in \cite[Theorem 5.1]{Zhou_2020_1}.
        In that paper, viscosity solutions were defined in terms of a family $\{D_{M, \mu}\}_{M > 0, \mu > 0}$ such that, for any numbers $M > 0$ and $\mu > 0$, the set $D_{M, \mu}$ consists of all functions $x \colon [- h, T] \to \mathbb{R}^n$ such that $\|x(\cdot)\|_\infty \leq M$ and
        \begin{equation*}
            \sup \biggl\{ \frac{\|x(\xi_1) - x(\xi_2)\|}{|\xi_1 - \xi_2|}
            \colon \xi_1, \xi_2 \in [- h, T], \ \xi_1 \neq \xi_2 \biggr\}
            \leq \mu (1 + M).
        \end{equation*}
        It was assumed that the Hamiltonian $H$ is of so-called Bellman form
        \begin{equation*}
            H(t, x(\cdot), s)
            \coloneq \inf_{u \in U} \bigl( \langle f(t, x(\cdot), u), s \rangle + g(t, x(\cdot), u) \bigr)
            \quad \forall t \in [0, T], \ x(\cdot) \in C, \ s \in \mathbb{R}^n,
        \end{equation*}
        where $U$ is a metric space, the mappings $(f, g) \colon [0, T] \times C \times U \to \mathbb{R}^n \times \mathbb{R}$ are continuous, and there exists a number $L_{f, g} > 0$ such that
        \begin{align*}
            \|f(t, x(\cdot), u) - f(t, y(\cdot), u)\|
            \vee |g(t, x(\cdot), u) - g(t, y(\cdot), u)|
            & \leq L_{f, g} \|x(\cdot \wedge t) - y(\cdot \wedge t)\|_\infty, \\
            \|f(t, x(\cdot), u)\| \vee |g(t, x(\cdot), u)|
            & \leq L_{f, g} (1 + \|x(\cdot \wedge t)\|_\infty)
        \end{align*}
        for all $t \in [0, T]$, $x(\cdot)$, $y(\cdot) \in C$, $u \in U$.
        Uniqueness for viscosity solutions was established in the class of functionals $\varphi$ satisfying the following global growth and Lipschitz conditions:
        there exists a number $L_\varphi > 0$ such that
        \begin{align*}
            |\varphi(t, x(\cdot))|
            & \leq L_\varphi (1 + \|x(\cdot \wedge t)\|_\infty), \\
            |\varphi(t, x(\cdot)) - \varphi(\tau, y(\cdot))|
            & \leq L_\varphi (1 + \|x(\cdot \wedge t)\|_\infty \vee \|y(\cdot \wedge \tau)\|_\infty) \\
            & \quad \times (|t - \tau| + \|x(\cdot \wedge t) - y(\cdot \wedge \tau)\|_\infty)
        \end{align*}
        for all $t$, $\tau \in [0, T]$, $x(\cdot)$, $y(\cdot) \in C$.
        Thus, compared to \cite[Theorem 5.1]{Zhou_2020_1}, Theorems \ref{theorem_global_uniqueness_1} and \ref{theorem_global_uniqueness_2} guarantee uniqueness under the weaker assumptions on the Hamiltonian $H$ and in the wider classes of functionals $\varphi$.
        Furthermore, the presented in Section \ref{section_proofs} proofs of Theorems \ref{theorem_local_comparison_1} and \ref{theorem_local_comparison_2}, from which Theorems \ref{theorem_global_uniqueness_1} and \ref{theorem_global_uniqueness_2} follow directly, appear simpler and more straightforward than the proof of \cite[Theorem 5.1]{Zhou_2020_1}.
        In addition, note that the definition of a viscosity solution adopted in \cite{Zhou_2020_1} slightly differs from Definitions \ref{definition_viscosity_D} and \ref{definition_viscosity_D_p}.
        Namely, the maximum and minimum in \eqref{subsolution_condition} and \eqref{supersolution_condition} respectively are considered over the set $[t, T] \times D_{M, \mu}$ (instead of $[0, T] \times D_{M, \mu}$) and it is additionally required that, for the point $(t, x(\cdot))$ from \eqref{subsolution_condition} and \eqref{supersolution_condition}, the inequality $\|x(t)\| < M$ is fulfilled.
    \end{remark}

\section{Penalty functional}
\label{section_functional}

    This section is devoted to the construction of a special functional used in the proofs of Theorems \ref{theorem_local_comparison_1} and \ref{theorem_local_comparison_2} as a penalty functional (see Section \ref{section_proofs}).
    This functional, denoted by $V^L$ in what follows, depends on the parameter $L \geq 0$, which we fix from now on.

    The construction is carried out in four steps.

    \smallskip

    \textit{Step 1.}
    Consider mappings $(V_1, P_1, Q_1) \colon [0, T] \times C \times [0, T] \times C \to \mathbb{R} \times \mathbb{R} \times \mathbb{R}^n$ that are defined for all $(t, x(\cdot), \tau, y(\cdot)) \in [0, T] \times C \times [0, T] \times C$ by
    \begin{equation} \label{V_1}
        \begin{aligned}
            V_1(t, x(\cdot), \tau, y(\cdot))
            & \coloneq (2 L^2 + 1) (t - \tau)^2
            + 2 \|x(t) - y(\tau)\|^2, \\
            P_1(t, x(\cdot), \tau, y(\cdot))
            & \coloneq 2 (2 L^2 + 1) (t - \tau), \\
            Q_1(t, x(\cdot), \tau, y(\cdot))
            & \coloneq 4 (x(t) - y(\tau)).
        \end{aligned}
    \end{equation}
    Note that the mappings $(V_1, P_1, Q_1)$ are continuous, the functional $V_1$ is non-negative, and, for every $(t, x(\cdot), \tau, y(\cdot)) \in [0, T] \times C \times [0, T] \times C$,
    \begin{equation} \label{V_1_symmetry}
        \begin{aligned}
            V_1(t, x(\cdot), \tau, y(\cdot))
            & = V_1(\tau, y(\cdot), t, x(\cdot)), \\
            (P_1, Q_1)(t, x(\cdot), \tau, y(\cdot))
            & = - (P_1, Q_1)(\tau, y(\cdot), t, x(\cdot))
        \end{aligned}
    \end{equation}
    and
    \begin{equation} \label{V_1_non-anticipative}
        (V_1, P_1, Q_1)(t, x(\cdot), \tau, y(\cdot))
        = (V_1, P_1, Q_1)(t, x(\cdot \wedge t), \tau, y(\cdot \wedge \tau)).
    \end{equation}
    In addition, it can be directly verified (see also, e.g., \cite[p. 257]{Gomoyunov_Lukoyanov_RMS_2024}) that, for every point $(\hat{\tau}, \hat{y}(\cdot)) \in [0, T] \times C$, the functional
    \begin{equation} \label{hat_v_1}
        \hat{v}_1(t, x(\cdot))
        \coloneq V_1(t, x(\cdot), \hat{\tau}, \hat{y}(\cdot))
        \quad \forall (t, x(\cdot)) \in [0, T] \times C
    \end{equation}
    is $ci$-smooth and, for all $(t, x(\cdot)) \in [0, T) \times C$,
    \begin{equation} \label{V_1_derivatives}
        \partial_t \hat{v}_1(t, x(\cdot))
        = P_1 (t, x(\cdot), \hat{\tau}, \hat{y}(\cdot)),
        \quad \nabla_{x(\cdot)} \hat{v}_1(t, x(\cdot))
        = Q_1 (t, x(\cdot), \hat{\tau}, \hat{y}(\cdot)).
    \end{equation}

    \smallskip

    \textit{Step 2.}
    Further, consider a functional $V_2 \colon [0, T] \times C \times [0, T] \times C \to \mathbb{R}$ defined for all $(t, x(\cdot), \tau, y(\cdot)) \in [0, T] \times C \times [0, T] \times C$ by
    \begin{equation} \label{V_2}
        V_2(t, x(\cdot), \tau, y(\cdot))
        \coloneq V_1(t, x(\cdot), \tau, y(\cdot))
        \vee \|x(\cdot \wedge t \wedge \tau) - y(\cdot \wedge t \wedge \tau) \|_\infty^2.
    \end{equation}
    Here, in accordance with the adopted notation,
    \begin{equation*}
        \|x(\cdot \wedge t \wedge \tau) - y(\cdot \wedge t \wedge \tau) \|_\infty
        = \max_{\xi \in [- h, t \wedge \tau]} \|x(\xi) - y(\xi)\|.
    \end{equation*}
    Note that the functional $V_2$ is continuous and non-negative and
    \begin{equation*}
        V_2(t, x(\cdot), \tau, y(\cdot))
        = V_2(\tau, y(\cdot), t, x(\cdot)),
        \quad V_2(t, x(\cdot), \tau, y(\cdot))
        = V_2(t, x(\cdot \wedge t), \tau, y(\cdot \wedge \tau))
    \end{equation*}
    for all $(t, x(\cdot), \tau, y(\cdot)) \in [0, T] \times C \times [0, T] \times C$.

    \begin{proposition} \label{proposition_V_2}
        The following statements hold.
        \begin{itemize}
            \item[\rm (i)]
            For every $(t, x(\cdot), \tau, y(\cdot)) \in [0, T] \times C \times [0, T] \times C$, the equality $V_2(t, x(\cdot), \tau, y(\cdot)) = 0$ takes place if and only if $(t, x(\cdot \wedge t)) = (\tau, y(\cdot \wedge \tau))$.

            \item[\rm (ii)]
            For every $(t, x(\cdot), \tau, y(\cdot)) \in [0, T] \times C \times [0, T] \times C$ with
            \begin{equation*}
                \tau
                \geq t,
                \quad \|y(\tau) - y(\xi)\|
                \leq L (\tau - t)
                \quad \forall \xi \in [t, \tau],
            \end{equation*}
            the inequality below is valid:
            \begin{equation} \label{proposition_V_2_main}
                V_2(t, x(\cdot), \tau, y(\cdot))
                \geq \|x(\cdot \wedge t) - y(\cdot \wedge \tau)\|_\infty^2.
            \end{equation}
        \end{itemize}
    \end{proposition}
    \begin{proof}
        (i)
            In the case where $V_2(t, x(\cdot), \tau, y(\cdot)) = 0$, we have
            \begin{equation*}
                V_1(t, x(\cdot), \tau, y(\cdot))
                = 0,
                \quad \|x(\cdot \wedge t \wedge \tau) - y(\cdot \wedge t \wedge \tau) \|_\infty
                = 0,
            \end{equation*}
            which imply that $t = \tau$ and $\|x(\cdot \wedge t) - y(\cdot \wedge \tau)\|_\infty = 0$.
            Hence, $(t, x(\cdot \wedge t)) = (\tau, y(\cdot \wedge \tau))$.
            On the other hand, if $(t, x(\cdot \wedge t)) = (\tau, y(\cdot \wedge \tau))$,
            we get
            \begin{equation*}
                V_1(t, x(\cdot), \tau, y(\cdot))
                = 0,
                \quad \|x(\cdot \wedge t \wedge \tau) - y(\cdot \wedge t \wedge \tau)\|_\infty
                = \|x(\cdot \wedge t) - y(\cdot \wedge \tau)\|_\infty
                = 0.
            \end{equation*}
            Then, $V_2(t, x(\cdot), \tau, y(\cdot)) = 0$.

        (ii)
            Since $\tau \geq t$,
            \begin{equation*}
                \|x(\cdot \wedge t) - y(\cdot \wedge \tau)\|_\infty^2
                = \|x(\cdot \wedge t \wedge \tau) - y(\cdot \wedge t \wedge \tau) \|_\infty^2
                \vee \max_{\xi \in [t, \tau]} \|x(t) - y(\xi)\|^2.
            \end{equation*}
            For every $\xi \in [t, \tau]$, we derive
            \begin{equation*}
                \|x(t) - y(\xi)\|
                \leq \|x(t) - y(\tau)\| + \|y(\tau) - y(\xi)\|
                \leq \|x(t) - y(\tau)\| + L (\tau - t),
            \end{equation*}
            and, therefore, recalling definition \eqref{V_1} of the functional $V_1$, we get
            \begin{align*}
                \|x(t) - y(\xi)\|^2
                & \leq (\|x(t) - y(\tau)\| + L (\tau - t))^2 \\
                & \leq 2 \|x(t) - y(\tau)\|^2 + 2 L^2 (\tau - t)^2 \\
                & \leq V_1(t, x(\cdot), \tau, y(\cdot)).
            \end{align*}
            Thus, we come to the desired inequality \eqref{proposition_V_2_main} and complete the proof.
    \end{proof}

    \smallskip

    \textit{Step 3.}
    Now, consider mappings $(V_3, P_3, Q_3) \colon [0, T] \times C \times [0, T] \times C \to \mathbb{R} \times \mathbb{R} \times \mathbb{R}^n$ defined for all $(t, x(\cdot), \tau, y(\cdot)) \in [0, T] \times C \times [0, T] \times C$ by (see also \eqref{V_1})
    \begin{equation} \label{V_3_1}
        \begin{aligned}
            V_3(t, x(\cdot), \tau, y(\cdot))
            & \coloneq \frac{\bigl( V_2(t, x(\cdot), \tau, y(\cdot))
            - V_1(t, x(\cdot), \tau, y(\cdot)) \bigr)^2}
            {V_2(t, x(\cdot), \tau, y(\cdot))}, \\
            P_3(t, x(\cdot), \tau, y(\cdot))
            & \coloneq - 2 \frac{V_2(t, x(\cdot), \tau, y(\cdot))
            - V_1(t, x(\cdot), \tau, y(\cdot))}
            {V_2(t, x(\cdot), \tau, y(\cdot))}
             P_1(t, x(\cdot), \tau, y(\cdot)) \\
            & = - 4 (2 L^2 + 1) \frac{V_2(t, x(\cdot), \tau, y(\cdot))
            - V_1(t, x(\cdot), \tau, y(\cdot))}
            {V_2(t, x(\cdot), \tau, y(\cdot))}
            (t - \tau), \\
            Q_3(t, x(\cdot), \tau, y(\cdot))
            & \coloneq - 2 \frac{V_2(t, x(\cdot), \tau, y(\cdot))
            - V_1(t, x(\cdot), \tau, y(\cdot))}
            {V_2(t, x(\cdot), \tau, y(\cdot))}
            Q_1(t, x(\cdot), \tau, y(\cdot)) \\
            & = - 8 \frac{V_2(t, x(\cdot), \tau, y(\cdot))
            - V_1(t, x(\cdot), \tau, y(\cdot))}
            {V_2(t, x(\cdot), \tau, y(\cdot))}
            (x(t) - y(\tau))
        \end{aligned}
    \end{equation}
    if $(t, x(\cdot \wedge t)) \neq (\tau, y(\cdot \wedge \tau))$ and by
    \begin{equation} \label{V_3_2}
        (V_3, P_3, Q_3)(t, x(\cdot), \tau, y(\cdot))
        \coloneq (0, 0, 0)
    \end{equation}
    if $(t, x(\cdot \wedge t)) = (\tau, y(\cdot \wedge \tau))$.
    The mappings $(V_3, P_3, Q_3)$ are well-defined since, for every $(t, x(\cdot), \tau, y(\cdot)) \in [0, T] \times C \times [0, T] \times C$ such that $(t, x(\cdot \wedge t)) \neq (\tau, y(\cdot \wedge \tau))$, we have $V_2(t, x(\cdot), \tau, y(\cdot)) \neq 0$ by Proposition \ref{proposition_V_2}, (i).
    Note that the functional $V_3$ is non-negative, and, for every $(t, x(\cdot), \tau, y(\cdot)) \in [0, T] \times C \times [0, T] \times C$,
    \begin{equation} \label{V_3_symmetry}
        \begin{aligned}
            V_3(t, x(\cdot), \tau, y(\cdot))
            & = V_3(\tau, y(\cdot), t, x(\cdot)), \\
            (P_3, Q_3)(t, x(\cdot), \tau, y(\cdot))
            & = - (P_3, Q_3)(\tau, y(\cdot), t, x(\cdot))
        \end{aligned}
    \end{equation}
    and
    \begin{equation} \label{V_3_non-anticipative}
        (V_3, P_3, Q_3)(t, x(\cdot), \tau, y(\cdot))
        = (V_3, P_3, Q_3)(t, x(\cdot \wedge t), \tau, y(\cdot \wedge \tau)).
    \end{equation}
    Moreover, for every $(t, x(\cdot), \tau, y(\cdot)) \in [0, T] \times C \times [0, T] \times C$, the inequalities below hold:
    \begin{equation} \label{V_3_estimates}
        \begin{aligned}
            V_3(t, x(\cdot), \tau, y(\cdot))
            & \leq V_2(t, x(\cdot), \tau, y(\cdot)) \\
            |P_3(t, x(\cdot), \tau, y(\cdot))|
            & \leq 2 |P_1(t, x(\cdot), \tau, y(\cdot))|
            \leq 4 (2 L^2 + 1) |t - \tau|, \\
            \|Q_3(t, x(\cdot), \tau, y(\cdot))\|
            & \leq 2 \|Q_1(t, x(\cdot), \tau, y(\cdot))\|
            \leq 8 \|x(t) - y(\tau)\|.
        \end{aligned}
    \end{equation}
    Indeed, in the case where $(t, x(\cdot \wedge t)) = (\tau, y(\cdot \wedge \tau))$, these inequalities are valid due to \eqref{V_3_2}, and if $(t, x(\cdot \wedge t)) \neq (\tau, y(\cdot \wedge \tau))$, they follow from the inequalities (see \eqref{V_2})
    \begin{equation*}
        0
        \leq \frac{V_2(t, x(\cdot), \tau, y(\cdot)) - V_1(t, x(\cdot), \tau, y(\cdot))}{V_2(t, x(\cdot), \tau, y(\cdot))}
        \leq 1.
    \end{equation*}

    \begin{proposition} \label{proposition_V_3_continuity}
        The mappings $(V_3, P_3, Q_3)$ are continuous.
    \end{proposition}
    \begin{proof}
        Let a point $(t, x(\cdot), \tau, y(\cdot)) \in [0, T] \times C \times [0, T] \times C$ be fixed.
        In the case where $(t, x(\cdot \wedge t)) \neq (\tau, y(\cdot \wedge \tau))$, continuity of $(V_3, P_3, Q_3)$ at $(t, x(\cdot), \tau, y(\cdot))$ follows directly from continuity of the mappings $(V_1, P_1, Q_1)$ and $V_2$.
        Suppose that $(t, x(\cdot \wedge t)) = (\tau, y(\cdot \wedge \tau))$ and consider a sequence $\{(t_k, x_k(\cdot), \tau_k, y_k(\cdot))\}_{k \in \mathbb{N}} \subset [0, T] \times C \times [0, T] \times C$ such that
        \begin{equation*}
            \lim_{k \to \infty} \bigl( |t_k - t| + \|x_k(\cdot) - x(\cdot)\|_\infty
            + |\tau_k - \tau| + \|y_k(\cdot) - y(\cdot)\|_\infty \bigr)
            = 0.
        \end{equation*}

        In view of \eqref{V_3_estimates}, we derive
        \begin{equation*}
            0
            \leq V_3(t_k, x_k(\cdot), \tau_k, y_k(\cdot))
            \leq V_2(t_k, x_k(\cdot), \tau_k, y_k(\cdot))
        \end{equation*}
        for all $k \in \mathbb{N}$.
        By continuity of $V_2$ and according to Proposition \ref{proposition_V_2}, (i),
        \begin{equation*}
            \lim_{k \to \infty} V_2(t_k, x_k(\cdot), \tau_k, y_k(\cdot))
            = V_2(t, x(\cdot), \tau, y(\cdot))
            = 0.
        \end{equation*}
        Hence, we obtain (see \eqref{V_3_2})
        \begin{equation*}
            \lim_{k \to \infty} V_3(t_k, x_k(\cdot), \tau_k, y_k(\cdot))
            = 0
            = V_3(t, x(\cdot), \tau, y(\cdot)).
        \end{equation*}

        In a similar way, we have (see \eqref{V_3_estimates})
        \begin{align*}
            |P_3(t_k, x_k(\cdot), \tau_k, y_k(\cdot))|
            & \leq 2 |P_1(t_k, x_k(\cdot), \tau_k, y_k(\cdot))|, \\
            \| Q_3(t_k, x_k(\cdot), \tau_k, y_k(\cdot)) \|
            & \leq 2 \|Q_1(t_k, x_k(\cdot), \tau_k, y_k(\cdot))\|
        \end{align*}
        for all $k \in \mathbb{N}$.
        By continuity of $(P_1, Q_1)$ and since $(t, x(t)) = (\tau, y(\tau))$, we get (see \eqref{V_1})
        \begin{align*}
            \lim_{k \to \infty} |P_1(t_k, x_k(\cdot), \tau_k, y_k(\cdot))|
            & = |P_1(t, x(\cdot), \tau, y(\cdot))|
            = 0, \\
            \lim_{k \to \infty} \|Q_1(t_k, x_k(\cdot), \tau_k, y_k(\cdot))\|
            & = \|Q_1(t, x(\cdot), \tau, y(\cdot))\|
            = 0.
        \end{align*}
        Therefore, we derive (see \eqref{V_3_2})
        \begin{equation*}
            \lim_{k \to \infty} (P_3, Q_3)(t_k, x_k(\cdot), \tau_k, y_k(\cdot))
            = (0, 0)
            = (P_3, Q_3)(t, x(\cdot), \tau, y(\cdot)),
        \end{equation*}
        which concludes the proof.
    \end{proof}

    \begin{proposition} \label{proposition_V_3_differentiability}
        Fix $(\hat{\tau}, \hat{y}(\cdot)) \in [0, T] \times C$ and consider the functional
        \begin{equation*}
            \hat{v}_3(t, x(\cdot))
            \coloneq V_3(t, x(\cdot), \hat{\tau}, \hat{y}(\cdot))
            \quad \forall (t, x(\cdot)) \in [0, T] \times C.
        \end{equation*}
        Let a point $(t, x(\cdot)) \in [0, T) \times C$ be such that
        \begin{equation} \label{hat_V_3_case_1}
            t < \hat{\tau},
            \quad \|\hat{y}(\xi) - \hat{y}(t)\|
            \leq L (\xi - t)
            \quad \forall \xi \in [t, \hat{\tau}]
        \end{equation}
        or
        \begin{equation} \label{hat_V_3_case_2}
            t \geq \hat{\tau}.
        \end{equation}
        Then, $\hat{v}_3$ is $ci$-differentiable at the point $(t, x(\cdot))$ and
        \begin{equation} \label{hat_v_3_derivatives}
            \partial_t \hat{v}_3(t, x(\cdot))
            = P_3 (t, x(\cdot), \hat{\tau}, \hat{y}(\cdot)),
            \quad \nabla_{x(\cdot)} \hat{v}_3(t, x(\cdot))
            = Q_3 (t, x(\cdot), \hat{\tau}, \hat{y}(\cdot)).
        \end{equation}
    \end{proposition}
    \begin{proof}
        Consider the functional $\hat{v}_1$ from \eqref{hat_v_1} and denote
        \begin{equation*}
            \hat{v}_2(t, x(\cdot))
            \coloneq V_2(t, x(\cdot), \hat{\tau}, \hat{y}(\cdot))
            \quad \forall (t, x(\cdot)) \in [0, T] \times C.
        \end{equation*}

        Fix a point $(t, x(\cdot)) \in [0, T) \times C$ such that either condition \eqref{hat_V_3_case_1} or condition \eqref{hat_V_3_case_2} is satisfied.
        Take a function $z(\cdot) \in \Lip(t, x(\cdot))$ and let $L_z > 0$ be such that
        \begin{equation} \label{proof_L_z}
            \|z(\xi_1) - z(\xi_2)\|
            \leq L_z |\xi_1 - \xi_2|
            \quad \forall \xi_1, \xi_2 \in [t, T].
        \end{equation}

        \smallskip

        \textit{Case 1.}
        Suppose that
        \begin{equation} \label{proof_case_1}
            \hat{v}_1(t, x(\cdot))
            < c,
            \quad c
            \coloneq \|x(\cdot \wedge t \wedge \hat{\tau}) - \hat{y}(\cdot \wedge t \wedge \hat{\tau})\|_\infty^2.
        \end{equation}

        Let us show that there exists $\delta_1 \in (0, T - t]$ such that
        \begin{equation} \label{proof_case_1_eq_1}
            \|z(\cdot \wedge (t + \delta) \wedge \hat{\tau}) - \hat{y}(\cdot \wedge (t + \delta) \wedge \hat{\tau})\|_\infty^2
            = c
            \quad \forall \delta \in (0, \delta_1].
        \end{equation}
        In the case of condition \eqref{hat_V_3_case_1}, we have
        \begin{equation*}
            \|x(t) - \hat{y}(t)\|
            \leq \|x(t) - \hat{y}(\hat{\tau})\| + \|\hat{y}(\hat{\tau}) - \hat{y}(t)\|
            \leq \|x(t) - \hat{y}(\hat{\tau})\| + L (\hat{\tau} - t),
        \end{equation*}
        which, according to \eqref{proof_case_1}, implies that
        \begin{equation*}
            \|x(t) - \hat{y}(t)\|^2
            \leq 2 \|x(t) - \hat{y}(\hat{\tau})\|^2  + 2 L^2 (\hat{\tau} - t)^2
            \leq \hat{v}_1(t, x(\cdot))
            < c.
        \end{equation*}
        Therefore, we can choose $\delta_1 \in (0, T - t]$ such that $t + \delta_1 \leq \hat{\tau}$ and
        \begin{equation*}
            \|z(\xi) - \hat{y}(\xi)\|^2
            < c
            \quad \forall \xi \in [t, t + \delta_1].
        \end{equation*}
        Then, for every $\delta \in (0, \delta_1]$, recalling the definition of $c$ (see \eqref{proof_case_1}), we derive
        \begin{align*}
            & \|z(\cdot \wedge (t + \delta) \wedge \hat{\tau}) - \hat{y}(\cdot \wedge (t + \delta) \wedge \hat{\tau})\|_\infty^2 \\
            & \quad = \|x(\cdot \wedge t \wedge \hat{\tau}) - \hat{y}(\cdot \wedge t \wedge \hat{\tau})\|_\infty^2
            \vee \max_{\xi \in [t, t + \delta]} \|z(\xi) - \hat{y}(\xi)\|^2 \\
            & \quad = c.
        \end{align*}
        In the case of condition \eqref{hat_V_3_case_2}, equality \eqref{proof_case_1_eq_1} is satisfied for every $\delta \in (0, T - t]$.
        Indeed, observing that $(t + \delta) \wedge \hat{\tau} = \hat{\tau} = t \wedge \hat{\tau}$, we obtain
        \begin{align*}
            \|z(\cdot \wedge (t + \delta) \wedge \hat{\tau}) - \hat{y}(\cdot \wedge (t + \delta) \wedge \hat{\tau})\|_\infty^2
            & = \|z(\cdot \wedge t \wedge \hat{\tau}) - \hat{y}(\cdot \wedge t \wedge \hat{\tau})\|_\infty^2 \\
            & = \|x(\cdot \wedge t \wedge \hat{\tau}) - \hat{y}(\cdot \wedge t \wedge \hat{\tau})\|_\infty^2 \\
            & = c.
        \end{align*}
        So, we can simply take $\delta_1 \coloneq T - t$.

        Since the functional $\hat{v}_1$ is continuous and non-anticipative (see \eqref{V_1_non-anticipative}), we have
        \begin{equation} \label{v_1_convergence}
            \lim_{\delta \to 0^+} \hat{v}_1(t + \delta, z(\cdot))
            = \hat{v}_1(t, z(\cdot))
            = \hat{v}_1(t, z(\cdot \wedge t))
            = \hat{v}_1(t, x(\cdot \wedge t))
            = \hat{v}_1(t, x(\cdot)).
        \end{equation}
        Therefore, and owing to \eqref{proof_case_1}, there exists $\delta_2 \in (0, T - t]$ such that
        \begin{equation} \label{proof_case_1_ineq_2}
            \hat{v}_1(t + \delta, z(\cdot))
            < c
            \quad \forall \delta \in (0, \delta_2].
        \end{equation}

        Put $\delta_3 \coloneq \min\{\delta_1, \delta_2\}$.
        Then, due to \eqref{V_2}, \eqref{proof_case_1}, \eqref{proof_case_1_eq_1}, and \eqref{proof_case_1_ineq_2}, we get
        \begin{equation} \label{proof_v_2_c}
            \hat{v}_2(t, x(\cdot))
            = c,
            \quad \hat{v}_2(t + \delta, z(\cdot))
            = c
            \quad \forall \delta \in (0, \delta_3].
        \end{equation}

        Note that $c > 0$ by \eqref{proof_case_1} and since the functional $\hat{v}_1$ is non-negative.
        Consequently, $(t, x(\cdot \wedge t)) \neq (\hat{\tau}, \hat{y}(\cdot \wedge \hat{\tau}))$ and $(t + \delta, z(\cdot \wedge (t + \delta))) \neq (\hat{\tau}, \hat{y}(\cdot \wedge \hat{\tau}))$ for all $\delta \in (0, \delta_3]$ thanks to Proposition \ref{proposition_V_2}, (i).
        Hence (see \eqref{V_3_1}), for every $\delta \in (0, \delta_3]$,
        \begin{equation*}
            \hat{v}_3(t, x(\cdot))
            = \frac{\bigl( c - \hat{v}_1(t, x(\cdot)) \bigr)^2}{c},
            \quad \hat{v}_3(t + \delta, z(\cdot))
            = \frac{ \bigl( c - \hat{v}_1(t + \delta, z(\cdot)) \bigr)^2}{c},
        \end{equation*}
        which yields
        \begin{equation} \label{proof_v_3_difference}
            \begin{aligned}
                & \hat{v}_3(t + \delta, z(\cdot)) - \hat{v}_3(t, x(\cdot)) \\
                & \quad = \frac{1}{c}
                \bigl( \hat{v}_1(t + \delta, z(\cdot)) + \hat{v}_1(t, x(\cdot)) - 2 c \bigr)
                \bigl( \hat{v}_1(t + \delta, z(\cdot)) - \hat{v}_1(t, x(\cdot)) \bigr).
            \end{aligned}
        \end{equation}

        Recall that $\hat{v}_1$ is $ci$-differentiable at the point $(t, x(\cdot))$ and the $ci$-derivatives are given by \eqref{V_1_derivatives}.
        This means that
        \begin{align*}
            \lim_{\delta \to 0^+} \frac{1}{\delta} \bigl( & \hat{v}_1(t + \delta, z(\cdot)) - \hat{v}_1(t, x(\cdot)) \\
            & \quad - P_1(t, x(\cdot), \hat{\tau}, \hat{y}(\cdot)) \delta
            - \langle Q_1(t, x(\cdot), \hat{\tau}, \hat{y}(\cdot)), z(t + \delta) - x(t) \rangle \bigr)
            = 0.
        \end{align*}
        Then, based on \eqref{proof_v_3_difference} and using \eqref{proof_L_z} and \eqref{v_1_convergence}, we derive
        \begin{align*}
            & \lim_{\delta \to 0^+} \frac{1}{\delta} \biggl( \hat{v}_3(t + \delta, z(\cdot)) - \hat{v}_3(t, x(\cdot)) \\
            & \quad - 2 \frac{\hat{v}_1(t, x(\cdot)) - c}{c}
            \bigl( P_1(t, x(\cdot), \hat{\tau}, \hat{y}(\cdot)) \delta
            + \langle Q_1(t, x(\cdot), \hat{\tau}, \hat{y}(\cdot)), z(t + \delta) - x(t) \rangle \bigr) \biggr)
            = 0.
        \end{align*}
        Therefore, $\hat{v}_3$ is $ci$-differentiable at the point $(t, x(\cdot))$ and, in view of \eqref{V_3_1} and the first equality in \eqref{proof_v_2_c}, equalities \eqref{hat_v_3_derivatives} are valid.

        \smallskip

        \textit{Case 2.}
        Now, suppose that
        \begin{equation} \label{proof_case_2}
            \hat{v}_1(t, x(\cdot))
            \geq c,
        \end{equation}
        where $c$ is taken from \eqref{proof_case_1}.
        According to \eqref{V_2}, we have $\hat{v}_2(t, x(\cdot)) = \hat{v}_1(t, x(\cdot))$.
        Then, by \eqref{V_3_1} and \eqref{V_3_2},
        \begin{equation*}
            \hat{v}_3(t, x(\cdot))
            = 0,
            \quad (P_3, Q_3)(t, x(\cdot), \hat{\tau}, \hat{y}(\cdot))
            = (0, 0).
        \end{equation*}
        Thus, in order to complete the proof, it suffices to verify that
        \begin{equation} \label{proof_case_2_convergence}
            \lim_{\delta \to 0^+} \frac{\hat{v}_3(t + \delta, z(\cdot))}{\delta}
            = 0.
        \end{equation}

        Let us show that there exist $M > 0$ and $\delta_4 \in (0, T - t]$ such that, for every $\delta \in (0, \delta_4]$ with $\hat{v}_3(t + \delta, z(\cdot)) \neq 0$,
        \begin{equation} \label{proof_case_2_statement}
            \begin{aligned}
                & \|z(\cdot \wedge (t + \delta) \wedge \hat{\tau}) - \hat{y}(\cdot \wedge (t + \delta) \wedge \hat{\tau})\|_\infty^2
                - \hat{v}_1(t + \delta, z(\cdot)) \\
                & \quad \leq M \delta \|z(\cdot \wedge (t + \delta) \wedge \hat{\tau}) - \hat{y}(\cdot \wedge (t + \delta) \wedge \hat{\tau})\|_\infty.
            \end{aligned}
        \end{equation}

        First of all, note that, thanks to \eqref{proof_case_2} and the definition of $c$ (see \eqref{proof_case_1}),
        \begin{equation} \label{proof_case_2_basic_0}
            \begin{aligned}
                & \|z(\cdot \wedge (t + \delta) \wedge \hat{\tau}) - \hat{y}(\cdot \wedge (t + \delta) \wedge \hat{\tau})\|_\infty^2
                - \hat{v}_1(t + \delta, z(\cdot)) \\
                & \quad \leq \|z(\cdot \wedge (t + \delta) \wedge \hat{\tau}) - \hat{y}(\cdot \wedge (t + \delta) \wedge \hat{\tau})\|_\infty^2
                - \|x(\cdot \wedge t \wedge \hat{\tau}) - \hat{y}(\cdot \wedge t \wedge \hat{\tau})\|_\infty^2 \\
                & \qquad + \hat{v}_1(t, x(\cdot)) - \hat{v}_1(t + \delta, z(\cdot))
            \end{aligned}
        \end{equation}
        for all $\delta \in (0, T - t]$.
        In addition, for every $\delta \in (0, T - t]$ with $\hat{v}_3(t + \delta, z(\cdot)) \neq 0$, in view of \eqref{V_2} and \eqref{V_3_1}, we get
        \begin{equation} \label{proof_case_2_basic}
            \hat{v}_2(t + \delta, z(\cdot))
            = \|z(\cdot \wedge (t + \delta) \wedge \hat{\tau}) - \hat{y}(\cdot \wedge (t + \delta) \wedge \hat{\tau})\|_\infty^2
            > \hat{v}_1(t + \delta, z(\cdot)).
        \end{equation}

        Suppose that condition \eqref{hat_V_3_case_1} holds.
        Then, put $M \coloneq 6 L^2 + 2 L + 3 + 2 L_z^2 + 6 L_z$, take $\delta_4 \in (0, T - t]$ such that $t + 2 \delta_4 \leq \hat{\tau}$, and fix $\delta \in (0, \delta_4]$ with $\hat{v}_3(t + \delta, z(\cdot)) \neq 0$.
        Based on the equality $(t + \delta) \wedge \hat{\tau} = t + \delta$, we derive
        \begin{align*}
            & \|z(\cdot \wedge (t + \delta) \wedge \hat{\tau}) - \hat{y}(\cdot \wedge (t + \delta) \wedge \hat{\tau})\|_\infty^2
            - \|x(\cdot \wedge t \wedge \hat{\tau}) - \hat{y}(\cdot \wedge t \wedge \hat{\tau})\|_\infty^2 \\
            & \quad = \max_{\xi \in [- h, t + \delta]} \|z(\xi) - \hat{y}(\xi)\|^2
            - \max_{\xi \in [- h, t]} \|x(\xi) - \hat{y}(\xi)\|^2 \\
            & \quad \leq \max_{\xi \in [t, t + \delta]} \|z(\xi) - \hat{y}(\xi)\|^2
            - \|x(t) - \hat{y}(t)\|^2.
        \end{align*}
        For every $\xi \in [t, t + \delta]$, due to \eqref{hat_V_3_case_1} and \eqref{proof_L_z}, we obtain
        \begin{align*}
            & \|z(\xi) - \hat{y}(\xi)\|^2 - \|x(t) - \hat{y}(t)\|^2 \\
            & \quad = \bigl( \|z(\xi) - \hat{y}(\xi)\| - \|x(t) - \hat{y}(t)\| \bigr)
            (\|z(\xi) - \hat{y}(\xi)\| + \|x(t) - \hat{y}(t)\|) \\
            & \quad \leq 2 \|z(\xi) - \hat{y}(\xi) - x(t) + \hat{y}(t)\|
            \max_{\eta \in [- h, t + \delta]} \|z(\eta) - \hat{y}(\eta)\| \\
            & \quad \leq 2 (L_z + L) \delta
            \|z(\cdot \wedge (t + \delta) \wedge \hat{\tau}) - \hat{y}(\cdot \wedge (t + \delta) \wedge \hat{\tau})\|_\infty.
        \end{align*}
        Hence,
        \begin{equation} \label{proof_case_2_subcase_1_1}
            \begin{aligned}
                & \|z(\cdot \wedge (t + \delta) \wedge \hat{\tau}) - \hat{y}(\cdot \wedge (t + \delta) \wedge \hat{\tau})\|_\infty^2
                - \|x(\cdot \wedge t \wedge \hat{\tau}) - \hat{y}(\cdot \wedge t \wedge \hat{\tau})\|_\infty^2 \\
                & \quad \leq 2 (L_z + L) \delta
                \|z(\cdot \wedge (t + \delta) \wedge \hat{\tau}) - \hat{y}(\cdot \wedge (t + \delta) \wedge \hat{\tau})\|_\infty.
            \end{aligned}
        \end{equation}

        Further, since $\delta \leq \hat{\tau} - t - \delta$, we have
        \begin{equation*}
            (t - \hat{\tau})^2 - (t + \delta - \hat{\tau})^2
            = \delta (2 \hat{\tau} - 2 t - \delta)
            = 2 \delta (\hat{\tau} - t - \delta) + \delta^2
            \leq 3 \delta (\hat{\tau} - t - \delta)
        \end{equation*}
        and
        \begin{equation} \label{proof_case_2_subcase_1_1.5}
            \begin{aligned}
                & \|x(t) - \hat{y}(\hat{\tau})\|^2 - \|z(t + \delta) - \hat{y}(\hat{\tau})\|^2 \\
                & \quad = \bigl( \|x(t) - \hat{y}(\hat{\tau})\| - \|z(t + \delta) - \hat{y}(\hat{\tau})\| \bigr)
                \bigl( \|x(t) - \hat{y}(\hat{\tau})\| + \|z(t + \delta) - \hat{y}(\hat{\tau})\| \bigr) \\
                & \quad \leq \|x(t) - z(t + \delta)\|
                \bigl( \|x(t) - z(t + \delta)\|
                + 2 \|z(t + \delta) - \hat{y}(\hat{\tau})\| \bigr) \\
                & \quad \leq L_z^2 \delta^2 + 2 L_z \delta \|z(t + \delta) - \hat{y}(\hat{\tau})\| \\
                & \quad \leq L_z^2 \delta (\hat{\tau} - t - \delta) + 2 L_z \delta \|z(t + \delta) - \hat{y}(\hat{\tau})\|.
            \end{aligned}
        \end{equation}
        Therefore, noting that (see \eqref{V_1})
        \begin{equation*}
            \hat{\tau} - t - \delta
            \leq \sqrt{\hat{v}_1(t + \delta, z(\cdot))},
            \quad \|z(t + \delta) - \hat{y}(\hat{\tau})\|
            \leq \sqrt{\hat{v}_1(t + \delta, z(\cdot))}
        \end{equation*}
        and using \eqref{proof_case_2_basic}, we get
        \begin{equation} \label{proof_case_2_subcase_1_2}
            \begin{aligned}
                & \hat{v}_1(t, x(\cdot)) - \hat{v}_1(t + \delta, z(\cdot)) \\
                & \quad \leq (6 L^2 + 3 + 2 L_z^2) \delta (\hat{\tau} - t - \delta)
                + 4 L_z \delta \|z(t + \delta) - \hat{y}(\hat{\tau})\| \\
                & \quad \leq (6 L^2 + 3 + 2 L_z^2 + 4 L_z) \delta
                \|z(\cdot \wedge (t + \delta) \wedge \hat{\tau}) - \hat{y}(\cdot \wedge (t + \delta) \wedge \hat{\tau})\|_\infty.
            \end{aligned}
        \end{equation}
        From \eqref{proof_case_2_basic_0}, \eqref{proof_case_2_subcase_1_1}, and \eqref{proof_case_2_subcase_1_2}, we derive \eqref{proof_case_2_statement}.

        Now, suppose that condition \eqref{hat_V_3_case_2} holds.
        Put $M \coloneq 2 L_z^2 + 4 L_z$ and $\delta_4 \coloneq T - t$ and fix $\delta \in (0, \delta_4]$ with $\hat{v}_3(t + \delta, z(\cdot)) \neq 0$.
        Since $(t + \delta) \wedge \hat{\tau} = \hat{\tau} = t \wedge \hat{\tau}$, we have
        \begin{equation*}
            \|z(\cdot \wedge (t + \delta) \wedge \hat{\tau}) - \hat{y}(\cdot \wedge (t + \delta) \wedge \hat{\tau})\|_\infty^2
            = \|x(\cdot \wedge t \wedge \hat{\tau}) - \hat{y}(\cdot \wedge t \wedge \hat{\tau})\|_\infty^2,
        \end{equation*}
        and, therefore, according to \eqref{proof_case_2_basic_0},
        \begin{equation*}
            \|z(\cdot \wedge (t + \delta) \wedge \hat{\tau})
            - \hat{y}(\cdot \wedge (t + \delta) \wedge \hat{\tau})\|_\infty^2
            - \hat{v}_1(t + \delta, z(\cdot))
            \leq \hat{v}_1(t, x(\cdot)) - \hat{v}_1(t + \delta, z(\cdot)).
        \end{equation*}
        In view of the inequality $(t - \hat{\tau})^2 \leq (t + \delta - \hat{\tau})^2$, we obtain (see \eqref{V_1})
        \begin{equation*}
            \hat{v}_1(t, x(\cdot)) - \hat{v}_1(t + \delta, z(\cdot))
            \leq 2 \bigl( \|x(t) - \hat{y}(\hat{\tau})\|^2 - \|z(t + \delta) - \hat{y}(\hat{\tau})\|^2 \bigr),
        \end{equation*}
        and, similarly to \eqref{proof_case_2_subcase_1_1.5}, we get
        \begin{align*}
            \|x(t) - \hat{y}(\hat{\tau})\|^2 - \|z(t + \delta) - \hat{y}(\hat{\tau})\|^2
            & \leq L_z^2 \delta^2
            + 2 L_z \delta \|z(t + \delta) - \hat{y}(\hat{\tau})\| \\
            & \leq L_z^2 \delta (t + \delta - \hat{\tau})
            + 2 L_z \delta \|z(t + \delta) - \hat{y}(\hat{\tau})\|.
        \end{align*}
        Hence, noting that
        \begin{equation*}
            t + \delta - \hat{\tau}
            \leq \sqrt{\hat{v}_1(t + \delta, z(\cdot))},
            \quad \|z(t + \delta) - \hat{y}(\hat{\tau})\|
            \leq \sqrt{\hat{v}_1(t + \delta, z(\cdot))}
        \end{equation*}
        and using \eqref{proof_case_2_basic}, we derive
        \begin{align*}
            & \|z(\cdot \wedge (t + \delta) \wedge \hat{\tau})
            - \hat{y}(\cdot \wedge (t + \delta) \wedge \hat{\tau})\|_\infty^2
            - \hat{v}_1(t + \delta, z(\cdot)) \\
            & \quad \leq (2 L_z^2 + 4 L_z) \delta
            \sqrt{\hat{v}_1(t + \delta, z(\cdot))} \\
            & \quad \leq (2 L_z^2 + 4 L_z) \delta
            \|z(\cdot \wedge (t + \delta) \wedge \hat{\tau}) - \hat{y}(\cdot \wedge (t + \delta) \wedge \hat{\tau})\|_\infty,
        \end{align*}
        which concludes the proof of \eqref{proof_case_2_statement}.

        For every $\delta \in (0, \delta_4]$, the inequality below is valid:
        \begin{equation} \label{proof_case_2_ineq_v_3}
            \hat{v}_3(t + \delta, z(\cdot))
            \leq M^2 \delta^2.
        \end{equation}
        Indeed, if $\hat{v}_3(t + \delta, z(\cdot)) = 0$, inequality \eqref{proof_case_2_ineq_v_3} holds automatically.
        Let $\hat{v}_3(t + \delta, z(\cdot)) \neq 0$.
        Then, thanks to \eqref{proof_case_2_statement} and \eqref{proof_case_2_basic}, we have
        \begin{equation*}
            \hat{v}_3(t + \delta, z(\cdot))
            = \frac{ \bigl( \|z(\cdot \wedge (t + \delta) \wedge \hat{\tau}) - \hat{y}(\cdot \wedge (t + \delta) \wedge \hat{\tau})\|_\infty^2
            - \hat{v}_1(t + \delta, z(\cdot)) \bigr)^2}
            {\|z(\cdot \wedge (t + \delta) \wedge \hat{\tau}) - \hat{y}(\cdot \wedge (t + \delta) \wedge \hat{\tau})\|_\infty^2}
            \leq M^2 \delta^2.
        \end{equation*}

        From \eqref{proof_case_2_ineq_v_3} and since $\hat{v}_3$ is non-negative, we derive \eqref{proof_case_2_convergence}.
        The proof is complete.
    \end{proof}

    \smallskip

    \textit{Step 4.}
    Finally, let mappings $(V^L, P^L, Q^L) \colon [0, T] \times C \times [0, T] \times C \to \mathbb{R} \times \mathbb{R} \times \mathbb{R}^n$ be defined for all $(t, x(\cdot), \tau, y(\cdot)) \in [0, T] \times C \times [0, T] \times C$ by
    \begin{equation} \label{V^L}
        \begin{aligned}
            & (V^L, P^L, Q^L)(t, x(\cdot), \tau, y(\cdot)) \\
            & \quad \coloneq (V_3, P_3, Q_3)(t, x(\cdot), \tau, y(\cdot)) + 2 (V_1, P_1, Q_1)(t, x(\cdot), \tau, y(\cdot)).
        \end{aligned}
    \end{equation}
    In a more direct way (see \eqref{V_1}, \eqref{V_3_1}, and \eqref{V_3_2}), if $(t, x(\cdot \wedge t)) \neq (\tau, y(\cdot \wedge \tau))$,
    \begin{equation} \label{V^L_1}
        \begin{aligned}
            V^L(t, x(\cdot), \tau, y(\cdot))
            & = V_2(t, x(\cdot), \tau, y(\cdot))
            + \frac{V_1(t, x(\cdot), \tau, y(\cdot))^2}
            {V_2(t, x(\cdot), \tau, y(\cdot))}, \\
            P^L(t, x(\cdot), \tau, y(\cdot))
            & = 2 \frac{V_1(t, x(\cdot), \tau, y(\cdot))}
            {V_2(t, x(\cdot), \tau, y(\cdot))}
            P_1(t, x(\cdot), \tau, y(\cdot)) \\
            & = 4 (2 L^2 + 1)  \frac{V_1(t, x(\cdot), \tau, y(\cdot))}
            {V_2(t, x(\cdot), \tau, y(\cdot))}
            (t - \tau), \\
            Q^L(t, x(\cdot), \tau, y(\cdot))
            & = 2 \frac{V_1(t, x(\cdot), \tau, y(\cdot))}
            {V_2(t, x(\cdot), \tau, y(\cdot))}
            Q_1(t, x(\cdot), \tau, y(\cdot)) \\
            & = 8 \frac{V_1(t, x(\cdot), \tau, y(\cdot))}
            {V_2(t, x(\cdot), \tau, y(\cdot))} (x(t) - y(\tau))
        \end{aligned}
    \end{equation}
    and if $(t, x(\cdot \wedge t)) = (\tau, y(\cdot \wedge \tau))$,
    \begin{equation} \label{V^L_2}
        (V^L, P^L, Q^L)(t, x(\cdot), \tau, y(\cdot))
        = (0, 0, 0).
    \end{equation}

    Summarizing the results above in this section, we obtain the theorem below.
    \begin{theorem} \label{theorem_V^L}
        The following statements hold.
        \begin{itemize}
        \item[\rm (i)]
            The functional $V^L$ is non-negative, the mappings $(V^L, P^L, Q^L)$ are continuous, and the equalities
            \begin{equation} \label{V_L_symmetry}
                \begin{aligned}
                    V^L(t, x(\cdot), \tau, y(\cdot))
                    & = V^L(\tau, y(\cdot), t, x(\cdot)), \\
                    (P^L, Q^L)(t, x(\cdot), \tau, y(\cdot))
                    & = - (P^L, Q^L)(\tau, y(\cdot), t, x(\cdot))
                \end{aligned}
            \end{equation}
            and
            \begin{equation} \label{V_L_non-anticipative}
                (V^L, P^L, Q^L)(t, x(\cdot), \tau, y(\cdot))
                = (V^L, P^L, Q^L)(t, x(\cdot \wedge t), \tau, y(\cdot \wedge \tau))
            \end{equation}
            hold for all $(t, x(\cdot), \tau, y(\cdot)) \in [0, T] \times C \times [0, T] \times C$.

        \item[\rm (ii)]
            For every $(t, x(\cdot), \tau, y(\cdot)) \in [0, T] \times C \times [0, T] \times C$, the equality $V^L(t, x(\cdot), \tau, y(\cdot)) = 0$ takes place if and only if $(t, x(\cdot \wedge t)) = (\tau, y(\cdot \wedge \tau))$.

        \item[\rm (iii)]
            For every $(t, x(\cdot), \tau, y(\cdot)) \in [0, T] \times C \times [0, T] \times C$ satisfying
            \begin{equation} \label{V_L_lower_bound_condition_1}
                \tau
                \geq t,
                \quad \|y(\tau) - y(\xi)\|
                \leq L (\tau - t)
                \quad \forall \xi \in [t, \tau]
            \end{equation}
            or
            \begin{equation} \label{V_L_lower_bound_condition_2}
                t
                \geq \tau,
                \quad \|x(t) - x(\xi)\|
                \leq L (t - \tau)
                \quad \forall \xi \in [\tau, t],
            \end{equation}
            the inequalities below are valid:
            \begin{equation} \label{V_L_lower_bound}
                \begin{aligned}
                    V^L(t, x(\cdot), \tau, y(\cdot))
                    \geq \|x(\cdot \wedge t) - y(\cdot \wedge \tau)\|_\infty^2,
                    \quad V^L(t, x(\cdot), \tau, y(\cdot))
                    \geq (t - \tau)^2.
                \end{aligned}
            \end{equation}

        \item[\rm (iv)]
            For every $(t, x(\cdot), \tau, y(\cdot)) \in [0, T] \times C \times [0, T] \times C$, the inequalities below hold:
            \begin{equation} \label{P_L_Q_L_upper_bound}
                \begin{aligned}
                    |P^L(t, x(\cdot), \tau, y(\cdot))|
                    & \leq 4 (2 L^2 + 1) |t - \tau|, \\
                    \|Q^L(t, x(\cdot), \tau, y(\cdot))\|
                    & \leq 8 \|x(t) - y(\tau)\|.
                \end{aligned}
            \end{equation}

        \item[\rm (v)]
            Fix $(\hat{\tau}, \hat{y}(\cdot)) \in [0, T] \times C$ and consider the functional
            \begin{equation*}
                \hat{v}^L(t, x(\cdot))
                \coloneq V^L(t, x(\cdot), \hat{\tau}, \hat{y}(\cdot))
                \quad \forall (t, x(\cdot)) \in [0, T] \times C.
            \end{equation*}
            Let a point $(t, x(\cdot)) \in [0, T) \times C$ be such that $t < \hat{\tau}$ and $\|\hat{y}(\xi) - \hat{y}(t)\| \leq L (\xi - t)$ for all $\xi \in [t, \hat{\tau}]$ or $t \geq \hat{\tau}$.
            Then, $\hat{v}^L$ is $ci$-differentiable at the point $(t, x(\cdot))$ and
            \begin{equation*}
                \partial_t \hat{v}^L(t, x(\cdot))
                = P^L (t, x(\cdot), \hat{\tau}, \hat{y}(\cdot)),
                \quad \nabla_{x(\cdot)} \hat{v}^L(t, x(\cdot))
                = Q^L (t, x(\cdot), \hat{\tau}, \hat{y}(\cdot)).
            \end{equation*}

        \item[\rm (vi)]
            Fix $(\bar{t}, \bar{x}(\cdot)) \in [0, T] \times C$ and consider the functional
            \begin{equation*}
                \bar{v}^L(\tau, y(\cdot))
                \coloneq V^L(\bar{t}, \bar{x}(\cdot), \tau, y(\cdot))
                \quad \forall (\tau, y(\cdot)) \in [0, T] \times C.
            \end{equation*}
            Let a point $(\tau, y(\cdot)) \in [0, T) \times C$ be such that $\tau < \bar{t}$ and $\|\bar{x}(\xi) - \bar{x}(\tau)\| \leq L (\xi - \tau)$ for all $\xi \in [\tau, \bar{t}]$ or $\tau \geq \bar{t}$.
            Then, $\bar{v}^L$ is $ci$-differentiable at the point $(\tau, y(\cdot))$ and
            \begin{equation*}
                \partial_\tau \bar{v}^L(\tau, y(\cdot))
                = - P^L (\bar{t}, \bar{x}(\cdot), \tau, y(\cdot)),
                \quad \nabla_{y(\cdot)} \bar{v}^L(\tau, y(\cdot))
                = - Q^L (\bar{t}, \bar{x}(\cdot), \tau, y(\cdot)).
            \end{equation*}
        \end{itemize}
    \end{theorem}
    \begin{proof}
        (i)
            The functional $V^L$ is non-negative since $V_1$ and $V_3$ are non-negative.
            Continuity of the mappings $(V^L, P^L, Q^L)$ follows from continuity of $(V_1, P_1, Q_1)$ and $(V_3, P_3, Q_3)$ (see Proposition \ref{proposition_V_3_continuity}).
            Equalities \eqref{V_L_symmetry} follow from \eqref{V_1_symmetry} and \eqref{V_3_symmetry}, while equality \eqref{V_L_non-anticipative} follows from \eqref{V_1_non-anticipative} and \eqref{V_3_non-anticipative}.

        (ii)
            Let $V^L(t, x(\cdot), \tau, y(\cdot)) = 0$ and suppose that $(t, x(\cdot \wedge t)) \neq (\tau, y(\cdot \wedge \tau))$.
            Hence, due to \eqref{V^L_1}, we have $V_2(t, x(\cdot), \tau, y(\cdot)) = 0$.
            Consequently, by Proposition \ref{proposition_V_2}, (i), we obtain $(t, x(\cdot \wedge t)) = (\tau, y(\cdot \wedge \tau))$, a contradiction.
            On the other hand, if $(t, x(\cdot \wedge t)) = (\tau, y(\cdot \wedge \tau))$, then $V_1(t, x(\cdot), \tau, y(\cdot)) = 0$ and $V_3(t, x(\cdot), \tau, y(\cdot)) = 0$, yielding $V^L(t, x(\cdot), \tau, y(\cdot)) = 0$.

        (iii)
            Let condition \eqref{V_L_lower_bound_condition_1} be satisfied.
            If $(t, x(\cdot \wedge t)) = (\tau, y(\cdot \wedge \tau))$, inequalities \eqref{V_L_lower_bound} hold since $V^L$ is non-negative.
            Suppose that $(t, x(\cdot \wedge t)) \neq (\tau, y(\cdot \wedge \tau))$.
            Then, we derive (see \eqref{V^L_1})
            \begin{equation*}
                V^L(t, x(\cdot), \tau, y(\cdot))
                \geq V_2(t, x(\cdot), \tau, y(\cdot)),
            \end{equation*}
            which gives the first inequality in \eqref{V_L_lower_bound} by Proposition \ref{proposition_V_2}, (ii).
            Furthermore, we have (see \eqref{V_1} and \eqref{V^L})
            \begin{equation*}
                V^L(t, x(\cdot), \tau, y(\cdot))
                \geq 2 V_1(t, x(\cdot), \tau, y(\cdot))
                \geq 2 (2 L^2 + 1) (t - \tau)^2,
            \end{equation*}
            which implies that the second inequality in \eqref{V_L_lower_bound} is valid.
            It remains to note that the case of condition \eqref{V_L_lower_bound_condition_2} reduces to the case of condition \eqref{V_L_lower_bound_condition_1} thanks to \eqref{V_L_symmetry}.

        (iv)
            Inequalities \eqref{P_L_Q_L_upper_bound} follow from \eqref{V_2}, \eqref{V^L_1}, and \eqref{V^L_2}.

        (v)
            The statement follows from Proposition \ref{proposition_V_3_differentiability} and the fact that the functional $\hat{v}_1$ (see \eqref{hat_v_1}) is $ci$-differentiable at every point $(t, x(\cdot)) \in [0, T) \times C$ with the $ci$-de\-riv\-a\-tives given by \eqref{V_1_derivatives}.

        (vi)
            The statement follows from (v) if we take \eqref{V_L_symmetry} into account.

        The proof is complete.
    \end{proof}

    For convenience of further use in Section \ref{section_proofs}, we additionally formulate the following result, which is a direct corollary of Theorem \ref{theorem_V^L}.
    \begin{theorem} \label{theorem_V^L_2}
        Let $D \subset C$ be a non-empty set satisfying condition \eqref{D_property_L} and let mappings $(V^{L_D}, P^{L_D}, Q^{L_D})$ be defined by \eqref{V^L} with $L = L_D$.
        Then, the following statements hold.
        \begin{itemize}
            \item[\rm (i)]
            The functional $V^{L_D}$ is continuous, non-negative, and $V^{L_D}(t, x(\cdot), t, x(\cdot)) = 0$ for all $(t, x(\cdot)) \in [0, T] \times C$.

            \item[\rm (ii)]
            For every $(t, x(\cdot), \tau, y(\cdot)) \in [0, T] \times D \times [0, T] \times D$, the inequalities below are valid:
            \begin{equation*}
                \begin{aligned}
                    V^{L_D}(t, x(\cdot), \tau, y(\cdot))
                    \geq \|x(\cdot \wedge t) - y(\cdot \wedge \tau)\|_\infty^2,
                    \quad V^{L_D}(t, x(\cdot), \tau, y(\cdot))
                    \geq (t - \tau)^2.
                \end{aligned}
            \end{equation*}

            \item[\rm (iii)]
            For every $(\hat{t}, \hat{x}(\cdot), \hat{\tau}, \hat{y}(\cdot)) \in [0, T) \times D \times [0, T) \times D$, the functionals
            \begin{align*}
                \psi_1(t, x(\cdot))
                & \coloneq V^{L_D}(t, x(\cdot), \hat{\tau}, \hat{y}(\cdot))
                \quad \forall (t, x(\cdot)) \in [0, T] \times C, \\
                \psi_2(\tau, y(\cdot))
                & \coloneq V^{L_D}(\hat{t}, \hat{x}(\cdot), \tau, y(\cdot))
                \quad \forall (\tau, y(\cdot)) \in [0, T] \times C
            \end{align*}
            are $ci$-smooth,
            \begin{equation*}
                \partial_t \psi_1(\hat{t}, \hat{x}(\cdot))
                = - \partial_\tau \psi_2(\hat{\tau}, \hat{y}(\cdot)),
                \quad \nabla_{x(\cdot)} \psi_1(\hat{t}, \hat{x}(\cdot))
                = - \nabla_{y(\cdot)} \psi_2(\hat{\tau}, \hat{y}(\cdot)),
            \end{equation*}
            and
            \begin{equation*}
                \|\nabla_{x(\cdot)} \psi_1(\hat{t}, \hat{x}(\cdot))\|
                = \|\nabla_{y(\cdot)} \psi_2(\hat{\tau}, \hat{y}(\cdot))\|
                \leq 8 \|\hat{x}(\hat{t}) - \hat{y}(\hat{\tau})\|.
            \end{equation*}
        \end{itemize}
    \end{theorem}

    Note that, in the proofs presented in Section \ref{section_proofs}, given a non-empty set $D \subset C$ satisfying condition \eqref{D_property_L}, we will not use the specific definition of the functional $V^{L_D}$;
    only its properties (i)--(iii) listed in Theorem \ref{theorem_V^L_2} will be employed.

\section{Proofs}
\label{section_proofs}

    In this section, we present the proofs of Theorems \ref{theorem_local_comparison_1} and \ref{theorem_local_comparison_2}.

    \begin{proof}[Proof of Theorem \ref{theorem_local_comparison_1}]
        For convenience, we split the proof into five steps.

        \smallskip

        \textit{Step 1.}
        Arguing by contradiction, suppose that
        \begin{equation} \label{proof_main_b}
            b
            \coloneq \max \bigl\{ \varphi_1(t, x(\cdot)) - \varphi_2(t, x(\cdot))
            \colon (t, x(\cdot)) \in [0, T] \times D \bigr\}
            > 0.
        \end{equation}
        The maximum is achieved by continuity of $\varphi_1$ and $\varphi_2$ and by compactness of $D$.
        Put
        \begin{equation} \label{proof_main_alpha}
            \alpha
            \coloneq \frac{b}{4 T}
            > 0.
        \end{equation}

        For every $\varepsilon > 0$ and every $\delta > 0$, consider a functional $\Phi_\varepsilon^\delta \colon [0, T] \times D \times [0, T] \times D \to \mathbb{R}$ defined for all $(t, x(\cdot), \tau, y(\cdot)) \in [0, T] \times D \times [0, T] \times D$ by
        \begin{align*}
            \Phi_\varepsilon^\delta(t, x(\cdot), \tau, y(\cdot))
            & \coloneq \varphi_1(t, x(\cdot)) - \varphi_2(\tau, y(\cdot)) \\
            & \quad - \alpha (2 T - t - \tau)
            - \frac{(t - \tau)^2}{\delta}
            - \frac{V^{L_D}(t, x(\cdot), \tau, y(\cdot))}{\varepsilon}.
        \end{align*}
        Here, the functional $V^{L_D}$ is defined according to \eqref{V^L} with $L = L_D$, where $L_D$ is the number from \eqref{D_property_L}.
        Recalling that $V^{L_D}$ is continuous by Theorem \ref{theorem_V^L_2}, (i), we conclude that there exists a point $(t_\varepsilon^\delta, x_\varepsilon^\delta(\cdot), \tau_\varepsilon^\delta, y_\varepsilon^\delta(\cdot)) \in [0, T] \times D \times [0, T] \times D$
        such that
        \begin{equation*}
            \Phi_\varepsilon^\delta(t_\varepsilon^\delta, x_\varepsilon^\delta(\cdot),
            \tau_\varepsilon^\delta, y_\varepsilon^\delta(\cdot))
            = \max \bigl\{ \Phi_\varepsilon^\delta(t, x(\cdot), \tau, y(\cdot))
            \colon (t, x(\cdot)), (\tau, y(\cdot))
            \in [0, T] \times D \bigr\}.
        \end{equation*}

        \smallskip

        \textit{Step 2.}
        Let $\varepsilon > 0$, $\delta > 0$.
        For every point $(t, x(\cdot)) \in [0, T] \times D$, taking into account that $V^{L_D}(t, x(\cdot), t, x(\cdot)) = 0$ by Theorem \ref{theorem_V^L_2}, (i), we have
        \begin{equation*}
            \varphi_1(t, x(\cdot))
            - \varphi_2(t, x(\cdot))
            - 2 \alpha (T - t)
            = \Phi_\varepsilon^\delta(t, x(\cdot), t, x(\cdot))
            \leq \Phi_\varepsilon^\delta(t_\varepsilon^\delta, x_\varepsilon^\delta(\cdot),
            \tau_\varepsilon^\delta, y_\varepsilon^\delta(\cdot)).
        \end{equation*}
        Therefore, due to \eqref{proof_main_b} and \eqref{proof_main_alpha}, we obtain
        \begin{equation} \label{proof_basic}
            \frac{b}{2}
            \leq \Phi_\varepsilon^\delta(t_\varepsilon^\delta, x_\varepsilon^\delta(\cdot),
            \tau_\varepsilon^\delta, y_\varepsilon^\delta(\cdot)).
        \end{equation}

        From \eqref{proof_basic}, we derive
        \begin{equation*}
            \frac{(t_\varepsilon^\delta - \tau_\varepsilon^\delta)^2}{\delta}
            + \frac{V^{L_D}(t_\varepsilon^\delta, x_\varepsilon^\delta(\cdot),
            \tau_\varepsilon^\delta, y_\varepsilon^\delta(\cdot))}{\varepsilon}
            \leq \varphi_1(t_\varepsilon^\delta, x_\varepsilon^\delta(\cdot))
            - \varphi_2(\tau_\varepsilon^\delta, y_\varepsilon^\delta(\cdot)).
        \end{equation*}
        Hence, and since $ V^{L_D}(t_\varepsilon^\delta, x_\varepsilon^\delta(\cdot), \tau_\varepsilon^\delta, y_\varepsilon^\delta(\cdot)) \geq 0$ by Theorem \ref{theorem_V^L_2}, (i), the inequalities
        \begin{equation} \label{proof_basic_esimates_varepsilon_delta}
            (t_\varepsilon^\delta - \tau_\varepsilon^\delta)^2
            \leq c \delta,
            \quad V^{L_D}(t_\varepsilon^\delta, x_\varepsilon^\delta(\cdot),
            \tau_\varepsilon^\delta, y_\varepsilon^\delta(\cdot))
            \leq c \varepsilon
        \end{equation}
        are valid with (see \eqref{proof_main_b})
        \begin{equation*}
            c
            \coloneq \max \bigl\{ \varphi_1(t, x(\cdot)) - \varphi_2(\tau, y(\cdot))
            \colon (t, x(\cdot)), (\tau, y(\cdot)) \in [0, T] \times D \bigr\}
            > 0.
        \end{equation*}
        According to Theorem \ref{theorem_V^L_2}, (ii), the second inequality in \eqref{proof_basic_esimates_varepsilon_delta} implies that
        \begin{equation} \label{proof_main_norms_varepsilon}
             \|x_\varepsilon^\delta(\cdot \wedge t_\varepsilon^\delta)
             - y_\varepsilon^\delta(\cdot \wedge \tau_\varepsilon^\delta)\|_\infty^2
             \leq c \varepsilon,
             \quad (t_\varepsilon^\delta - \tau_\varepsilon^\delta)^2
             \leq c \varepsilon.
        \end{equation}

        Further, noting that $V^{L_D}(t_\varepsilon^\delta, x_\varepsilon^\delta(\cdot), t_\varepsilon^\delta, x_\varepsilon^\delta(\cdot)) = 0$ by Theorem \ref{theorem_V^L_2}, (i), we get
        \begin{align*}
            & \varphi_1(t_\varepsilon^\delta, x_\varepsilon^\delta(\cdot))
            - \varphi_2(t_\varepsilon^\delta, x_\varepsilon^\delta(\cdot))
            - 2 \alpha (T - t_\varepsilon^\delta) \\
            & \quad = \Phi_\varepsilon^\delta(t_\varepsilon^\delta, x_\varepsilon^\delta(\cdot),
            t_\varepsilon^\delta, x_\varepsilon^\delta(\cdot)) \\
            & \quad \leq \Phi_\varepsilon^\delta(t_\varepsilon^\delta, x_\varepsilon^\delta(\cdot),
            \tau_\varepsilon^\delta, y_\varepsilon^\delta(\cdot)) \\
            & \quad \leq \varphi_1(t_\varepsilon^\delta, x_\varepsilon^\delta(\cdot))
            - \varphi_2(\tau_\varepsilon^\delta, y_\varepsilon^\delta(\cdot))
            - \alpha (2 T - t_\varepsilon^\delta - \tau_\varepsilon^\delta)
            - \frac{V^{L_D}(t_\varepsilon^\delta, x_\varepsilon^\delta(\cdot),
            \tau_\varepsilon^\delta, y_\varepsilon^\delta(\cdot))}{\varepsilon},
        \end{align*}
        which gives
        \begin{equation} \label{proof_main_V^L_varepsilon}
            \frac{V^{L_D}(t_\varepsilon^\delta, x_\varepsilon^\delta(\cdot),
            \tau_\varepsilon^\delta, y_\varepsilon^\delta(\cdot))}{\varepsilon}
            \leq \varphi_2(t_\varepsilon^\delta, x_\varepsilon^\delta(\cdot))
            - \varphi_2(\tau_\varepsilon^\delta, y_\varepsilon^\delta(\cdot))
            + \alpha (\tau_\varepsilon^\delta - t_\varepsilon^\delta).
        \end{equation}

        For $i \in \{1, 2\}$, denote by $\omega_{\varphi_i, D} \colon [0, + \infty) \to [0, + \infty)$ the modulus of continuity of the functional $\varphi_i$ defined by
        \begin{align*}
            \omega_{\varphi_i, D}(\theta)
            & \coloneq \max \bigl\{ |\varphi_i(t, x(\cdot)) - \varphi_i(\tau, y(\cdot))|
            \colon \\
            & \qquad (t, x(\cdot)), (\tau, y(\cdot)) \in [0, T] \times D,
            \ |t - \tau| + \|x(\cdot) - y(\cdot)\|_\infty \leq \theta \bigr\}
            \quad \forall \theta \geq 0.
        \end{align*}
        Using \eqref{D_property_1} and \eqref{proof_main_norms_varepsilon} and recalling that the functionals $\varphi_1$ and $\varphi_2$ are non-anticipative, we obtain
        \begin{equation} \label{proof_main_varphi_1_difference}
            \begin{aligned}
                | \varphi_1(t_\varepsilon^\delta, x_\varepsilon^\delta(\cdot))
                - \varphi_1(\tau_\varepsilon^\delta, y_\varepsilon^\delta(\cdot)) |
                & = | \varphi_1(t_\varepsilon^\delta, x_\varepsilon^\delta(\cdot \wedge t_\varepsilon^\delta))
                - \varphi_1(\tau_\varepsilon^\delta, y_\varepsilon^\delta(\cdot \wedge \tau_\varepsilon^\delta)) | \\
                & \leq \omega_{\varphi_1, D}(|t_\varepsilon^\delta - \tau_\varepsilon^\delta|
                + \|x_\varepsilon^\delta(\cdot \wedge t_\varepsilon^\delta)
                - y_\varepsilon^\delta(\cdot \wedge \tau_\varepsilon^\delta)\|_\infty) \\
                & \leq \omega_{\varphi_1, D}(2 \sqrt{c} \sqrt{\varepsilon})
            \end{aligned}
        \end{equation}
        and, similarly,
        \begin{equation} \label{proof_main_varphi_2_difference}
            | \varphi_2(t_\varepsilon^\delta, x_\varepsilon^\delta(\cdot))
            - \varphi_2(\tau_\varepsilon^\delta, y_\varepsilon^\delta(\cdot))|
            \leq \omega_{\varphi_2, D}(2 \sqrt{c} \sqrt{\varepsilon}).
        \end{equation}
        Hence, by \eqref{proof_main_norms_varepsilon}, \eqref{proof_main_V^L_varepsilon}, \eqref{proof_main_varphi_2_difference}, and Theorem \ref{theorem_V^L_2}, (ii), we have
        \begin{equation} \label{proof_main_norms_varepsilon_improved}
            \frac{\|x_\varepsilon^\delta(\cdot \wedge t_\varepsilon^\delta)
            - y_\varepsilon^\delta(\cdot \wedge \tau_\varepsilon^\delta)\|_\infty^2}{\varepsilon}
            \leq \omega_{\varphi_2, D}(2 \sqrt{c} \sqrt{\varepsilon}) + \alpha \sqrt{c} \sqrt{\varepsilon}.
        \end{equation}

        \smallskip

        \textit{Step 3.}
        Let $\varepsilon > 0$ and $\delta > 0$.
        Since $V^{L_D}(t_\varepsilon^\delta, x_\varepsilon^\delta(\cdot), \tau_\varepsilon^\delta, y_\varepsilon^\delta(\cdot)) \geq 0$ by Theorem \ref{theorem_V^L_2}, (i), it follows from \eqref{proof_basic} that
        \begin{equation*}
            \frac{b}{2}
            \leq \varphi_1(t_\varepsilon^\delta, x_\varepsilon^\delta(\cdot))
            - \varphi_2(\tau_\varepsilon^\delta, y_\varepsilon^\delta(\cdot)).
        \end{equation*}
        Therefore, in view of the boundary condition \eqref{varphi_1_varphi_2_boundary} and due to \eqref{proof_main_varphi_1_difference} and \eqref{proof_main_varphi_2_difference},
        \begin{align*}
            \frac{b}{2}
            & \leq \varphi_1(t_\varepsilon^\delta, x_\varepsilon^\delta(\cdot))
            - \varphi_1(T, x_\varepsilon^\delta(\cdot))
            + \varphi_2(T, x_\varepsilon^\delta(\cdot))
            - \varphi_2(t_\varepsilon^\delta, x_\varepsilon^\delta(\cdot)) \\
            & \quad + \varphi_2(t_\varepsilon^\delta, x_\varepsilon^\delta(\cdot))
            - \varphi_2(\tau_\varepsilon^\delta, y_\varepsilon^\delta(\cdot)) \\
            & \leq \omega_{\varphi_1, D}(T - t_\varepsilon^\delta)
            + \omega_{\varphi_2, D}(T - t_\varepsilon^\delta)
            + \omega_{\varphi_2, D}(2 \sqrt{c} \sqrt{\varepsilon})
        \end{align*}
        and, similarly,
        \begin{equation*}
            \frac{b}{2}
            \leq \omega_{\varphi_1, D}(T - \tau_\varepsilon^\delta)
            + \omega_{\varphi_2, D}(T - \tau_\varepsilon^\delta)
            + \omega_{\varphi_1, D}(2 \sqrt{c} \sqrt{\varepsilon}).
        \end{equation*}
        Thus, choosing $\varepsilon_\ast > 0$ from the condition
        \begin{equation*}
            \omega_{\varphi_1, D}(2 \sqrt{c} \sqrt{\varepsilon_\ast})
            \vee \omega_{\varphi_2, D}(2 \sqrt{c} \sqrt{\varepsilon_\ast})
            \leq \frac{b}{4},
        \end{equation*}
        we conclude that $t_\varepsilon^\delta < T$ and $\tau_\varepsilon^\delta < T$ for all $\varepsilon \in (0, \varepsilon_\ast]$ and $\delta > 0$.

        \smallskip

        \textit{Step 4.}
        Let $\varepsilon \in (0, \varepsilon_\ast]$ and $\delta > 0$.
        We consider test functionals $\psi_1 \colon [0, T] \times C \to \mathbb{R}$ and $\psi_2 \colon [0, T] \times C \to \mathbb{R}$ defined for all $(t, x(\cdot))$, $(\tau, y(\cdot)) \in [0, T] \times C$ by
        \begin{align*}
            \psi_1(t, x(\cdot))
            & \coloneq \varphi_2(\tau_\varepsilon^\delta, y_\varepsilon^\delta(\cdot))
            + \alpha (2 T - t - \tau_\varepsilon^\delta)
            + \frac{(t - \tau_\varepsilon^\delta)^2}{\delta}
            + \frac{V^{L_D}(t, x(\cdot), \tau_\varepsilon^\delta, y_\varepsilon^\delta(\cdot))}{\varepsilon}, \\
            \psi_2(\tau, y(\cdot))
            & \coloneq \varphi_1(t_\varepsilon^\delta, x_\varepsilon^\delta(\cdot))
            - \alpha (2 T - t_\varepsilon^\delta - \tau)
            - \frac{(t_\varepsilon^\delta - \tau)^2}{\delta}
            - \frac{V^{L_D}(t_\varepsilon^\delta, x_\varepsilon^\delta(\cdot), \tau, y(\cdot))}{\varepsilon}.
        \end{align*}

        In accordance with Theorem \ref{theorem_V^L_2}, (iii), the functionals $\psi_1$ and $\psi_2$ are $ci$-smooth,
        \begin{equation} \label{proof_derivatives_of_test_functionals}
            \partial_t \psi_1(t_\varepsilon^\delta, x_\varepsilon^\delta(\cdot))
            - \partial_\tau \psi_2(\tau_\varepsilon^\delta, y_\varepsilon^\delta(\cdot))
            = - 2 \alpha,
            \quad \nabla_{x(\cdot)} \psi_1(t_\varepsilon^\delta, x_\varepsilon^\delta(\cdot))
            = \nabla_{y(\cdot)} \psi_2(\tau_\varepsilon^\delta, y_\varepsilon^\delta(\cdot)),
        \end{equation}
        and
        \begin{equation} \label{proof_gradient_bound}
            \|\nabla_{x(\cdot)} \psi_1(t_\varepsilon^\delta, x_\varepsilon^\delta(\cdot))\|
            \leq 8 \frac{\|x_\varepsilon^\delta(t_\varepsilon^\delta)
            - y_\varepsilon^\delta(\tau_\varepsilon^\delta)\|}{\varepsilon}.
        \end{equation}

        Note that, for every $(t, x(\cdot)) \in [0, T] \times D$,
        \begin{equation} \label{varphi_1_max}
            \begin{aligned}
                \varphi_1(t_\varepsilon^\delta, x_\varepsilon^\delta(\cdot))
                - \psi_1(t_\varepsilon^\delta, x_\varepsilon^\delta(\cdot))
                & = \Phi_\varepsilon^\delta(t_\varepsilon^\delta, x_\varepsilon^\delta(\cdot),
                \tau_\varepsilon^\delta, y_\varepsilon^\delta(\cdot)) \\
                & \geq \Phi_\varepsilon^\delta(t, x(\cdot),
                \tau_\varepsilon^\delta, y_\varepsilon^\delta(\cdot)) \\
                & = \varphi_1(t, x(\cdot))
                - \psi_1(t, x(\cdot)),
            \end{aligned}
        \end{equation}
        and, similarly, for every $(\tau, y(\cdot)) \in [0, T] \times D$,
        \begin{equation} \label{varphi_2_min}
            \begin{aligned}
                \varphi_2(\tau_\varepsilon^\delta, y_\varepsilon^\delta(\cdot))
                - \psi_2(\tau_\varepsilon^\delta, y_\varepsilon^\delta(\cdot))
                & = - \Phi_\varepsilon^\delta(t_\varepsilon^\delta, x_\varepsilon^\delta(\cdot),
                \tau_\varepsilon^\delta, y_\varepsilon^\delta(\cdot)) \\
                & \leq - \Phi_\varepsilon^\delta(t_\varepsilon^\delta, x_\varepsilon^\delta(\cdot),
                \tau, y(\cdot)) \\
                & = \varphi_2(\tau, y(\cdot))
                - \psi_2(\tau, y(\cdot)).
            \end{aligned}
        \end{equation}

        Then, recalling that $\varphi_1$ is a viscosity $D$-subsolution and $\varphi_2$ is a viscosity $D$-supersolution of \eqref{HJ}, we derive
        \begin{align*}
            \partial_t \psi_1(t_\varepsilon^\delta, x_\varepsilon^\delta(\cdot))
            + H \bigl( t_\varepsilon^\delta, x_\varepsilon^\delta(\cdot),
            \nabla_{x(\cdot)} \psi_1(t_\varepsilon^\delta, x_\varepsilon^\delta(\cdot)) \bigr)
            & \geq 0, \\
            \partial_\tau \psi_2(\tau_\varepsilon^\delta, y_\varepsilon^\delta(\cdot))
            + H \bigl( \tau_\varepsilon^\delta, y_\varepsilon^\delta(\cdot),
            \nabla_{y(\cdot)} \psi_2(\tau_\varepsilon^\delta, y_\varepsilon^\delta(\cdot)) \bigr)
            & \leq 0.
        \end{align*}
        As a result, and due to \eqref{proof_derivatives_of_test_functionals}, we have
        \begin{equation} \label{contradiction_basic}
            H \bigl( t_\varepsilon^\delta, x_\varepsilon^\delta(\cdot),
            \nabla_{x(\cdot)} \psi_1(t_\varepsilon^\delta, x_\varepsilon^\delta(\cdot)) \bigr)
            - H\bigl( \tau_\varepsilon^\delta, y_\varepsilon^\delta(\cdot), \nabla_{x(\cdot)} \psi_1(t_\varepsilon^\delta, x_\varepsilon^\delta(\cdot)) \bigr)
            \geq 2 \alpha.
        \end{equation}

        \smallskip

        \textit{Step 5.}
        Take the number $L_{H, D}$ from assumption $(A.2)$ and fix $\varepsilon \in (0, \varepsilon_\ast]$ such that
        \begin{equation*}
            L_{H, D} \bigl( 8 \omega_{\varphi_2, D}(2 \sqrt{c} \sqrt{\varepsilon})
            + (8 \alpha + 1) \sqrt{c} \sqrt{\varepsilon} \bigr)
            \leq \frac{\alpha}{2}.
        \end{equation*}
        Using assumption $(A.1)$, choose $\delta > 0$ such that
        \begin{equation} \label{H_difference}
            |H(t, x(\cdot), s) - H(\tau, x(\cdot), s)|
            \leq \frac{\alpha}{2}
        \end{equation}
        for all $t$, $\tau \in [0, T]$ with $|t - \tau| \leq \sqrt{c} \sqrt{\delta}$ and all $x(\cdot) \in D$, $s \in B(8 \sqrt{c} / \sqrt{\varepsilon})$.

        Let us prove that
        \begin{equation} \label{contradiction_1}
            H \bigl( t_\varepsilon^\delta, x_\varepsilon^\delta(\cdot),
            \nabla_{x(\cdot)} \psi_1(t_\varepsilon^\delta, x_\varepsilon^\delta(\cdot)) \bigr)
            - H\bigl( \tau_\varepsilon^\delta, y_\varepsilon^\delta(\cdot), \nabla_{x(\cdot)} \psi_1(t_\varepsilon^\delta, x_\varepsilon^\delta(\cdot)) \bigr)
            \leq \alpha.
        \end{equation}
        Suppose that $t_\varepsilon^\delta \geq \tau_\varepsilon^\delta$.
        Then, due to \eqref{proof_main_norms_varepsilon}, \eqref{proof_main_norms_varepsilon_improved}, and \eqref{proof_gradient_bound}, we obtain
        \begin{equation} \label{contradiction_1_proof_1}
            \begin{aligned}
                & H \bigl( \tau_\varepsilon^\delta, x_\varepsilon^\delta(\cdot),
                \nabla_{x(\cdot)} \psi_1(t_\varepsilon^\delta, x_\varepsilon^\delta(\cdot)) \bigr)
                - H\bigl( \tau_\varepsilon^\delta, y_\varepsilon^\delta(\cdot), \nabla_{x(\cdot)} \psi_1(t_\varepsilon^\delta, x_\varepsilon^\delta(\cdot)) \bigr) \\
                & \quad \leq L_{H, D} \bigl( 1 + \|\nabla_{x(\cdot)} \psi_1(t_\varepsilon^\delta, x_\varepsilon^\delta(\cdot))\| \bigr)
                \|x_\varepsilon^\delta(\cdot \wedge \tau_\varepsilon^\delta)
                - y_\varepsilon^\delta(\cdot \wedge \tau_\varepsilon^\delta)\|_\infty \\
                & \quad \leq L_{H, D} \biggl( 1 + 8 \frac{\|x_\varepsilon^\delta(t_\varepsilon^\delta)
                - y_\varepsilon^\delta(\tau_\varepsilon^\delta)\|}{\varepsilon} \biggr)
                \|x_\varepsilon^\delta(\cdot \wedge t_\varepsilon^\delta)
                - y_\varepsilon^\delta(\cdot \wedge \tau_\varepsilon^\delta)\|_\infty \\
                & \quad \leq L_{H, D} \biggl( \|x_\varepsilon^\delta(\cdot \wedge t_\varepsilon^\delta)
                - y_\varepsilon^\delta(\cdot \wedge \tau_\varepsilon^\delta)\|_\infty
                + 8 \frac{\|x_\varepsilon^\delta(\cdot \wedge t_\varepsilon^\delta)
                - y_\varepsilon^\delta(\cdot \wedge \tau_\varepsilon^\delta)\|_\infty^2}{\varepsilon} \biggr) \\
                & \quad \leq L_{H, D} \bigl( 8 \omega_{\varphi_2, D}(2 \sqrt{c} \sqrt{\varepsilon})
                + (8 \alpha + 1) \sqrt{c} \sqrt{\varepsilon} \bigr) \\
                & \quad \leq \frac{\alpha}{2}.
            \end{aligned}
        \end{equation}
        Moreover, thanks to the inclusion $\nabla_{x(\cdot)} \psi_1(t_\varepsilon^\delta, x_\varepsilon^\delta(\cdot)) \in B(8 \sqrt{c} / \sqrt{\varepsilon})$ (see the first inequality in \eqref{proof_main_norms_varepsilon} and \eqref{proof_gradient_bound}) and the first inequality in \eqref{proof_basic_esimates_varepsilon_delta}, we get
        \begin{equation} \label{contradiction_1_proof_2}
            H \bigl( t_\varepsilon^\delta, x_\varepsilon^\delta(\cdot),
            \nabla_{x(\cdot)} \psi_1(t_\varepsilon^\delta, x_\varepsilon^\delta(\cdot)) \bigr)
            - H \bigl( \tau_\varepsilon^\delta, x_\varepsilon^\delta(\cdot),
            \nabla_{x(\cdot)} \psi_1(t_\varepsilon^\delta, x_\varepsilon^\delta(\cdot)) \bigr)
            \leq \frac{\alpha}{2}.
        \end{equation}
        Inequality \eqref{contradiction_1} follows from \eqref{contradiction_1_proof_1} and \eqref{contradiction_1_proof_2}.
        The case where $t_\varepsilon^\delta < \tau_\varepsilon^\delta$ can be handled similarly.

        According to \eqref{contradiction_basic} and \eqref{contradiction_1}, we have $\alpha \leq 0$, which contradicts \eqref{proof_main_alpha}.
        The proof of Theorem \ref{theorem_local_comparison_1} is complete.
    \end{proof}

    \begin{lemma} \label{lemma}
        Let $D \subset C$ be a non-empty set satisfying \eqref{D_property_v}, let $\varphi \colon [0, T] \times D \to \mathbb{R}$ be a functional such that $L_{\varphi, D} < + \infty$, and let $\psi \colon [0, T] \times C \to \mathbb{R}$ be a functional $ci$-dif\-fer\-en\-ti\-a\-ble at a point $(t, x(\cdot)) \in [0, T) \times D$.
        Then, the following statements hold.
        \begin{itemize}
            \item[\rm (i)]
                If condition \eqref{subsolution_condition} is fulfilled, the inequality below is valid:
                \begin{equation} \label{lemma_1}
                    \partial_t \psi(t, x(\cdot))
                    + \langle \nabla_{x(\cdot)} \psi(t, x(\cdot)), v \rangle
                    \geq - L_{\varphi, D}
                    \quad \forall v \in B(1).
                \end{equation}

            \item[\rm (ii)]
                If condition \eqref{supersolution_condition} is fulfilled, the inequality below is valid:
                \begin{equation*}
                    \partial_t \psi(t, x(\cdot))
                    + \langle \nabla_{x(\cdot)} \psi(t, x(\cdot)), v \rangle
                    \leq L_{\varphi, D}
                    \quad \forall v \in B(1).
                \end{equation*}
        \end{itemize}
    \end{lemma}
    \begin{proof}
        We proof (i) only, since the proof for (ii) is similar.

        Fix $v \in B(1)$, denote $z(\cdot) \coloneq z_{t, x(\cdot), v}(\cdot)$ (see \eqref{z_v}), and observe that $z(\cdot) \in D$ by \eqref{D_property_v}.
        Based on \eqref{subsolution_condition} and \eqref{L_varphi_definition}, we derive
        \begin{equation} \label{lemma_1_proof}
            \frac{\psi(t + \delta, z(\cdot))
            - \psi(t, x(\cdot))}{\delta}
            \geq \frac{\varphi(t + \delta, z(\cdot))
            - \varphi(t, x(\cdot))}{\delta}
            \geq - L_{\varphi, D}
            \quad \forall \delta \in (0, T - t].
        \end{equation}
        Since $\psi$ is $ci$-differentiable at $(t, x(\cdot))$ and $z(t + \delta) - x(t) = v \delta$ for all $\delta \in (0, T - t]$, we have
        \begin{equation*}
            \lim_{\delta \to 0^+} \frac{\psi(t + \delta, z(\cdot))
            - \psi(t, x(\cdot))}{\delta}
            = \partial_t \psi(t, x(\cdot))
            + \langle \nabla_{x(\cdot)} \psi(t, x(\cdot)), v \rangle.
        \end{equation*}
        Passing to the limit as $\delta \to 0^+$ in \eqref{lemma_1_proof}, we obtain \eqref{lemma_1}.
        The proof is complete.
    \end{proof}

    \begin{proof}[Proof of Theorem \ref{theorem_local_comparison_2}]
        Repeat Steps 1--4 of the proof of Theorem \ref{theorem_local_comparison_1}.

        Let $\varepsilon \in (0, \varepsilon_\ast]$ and $\delta > 0$.
        Thanks to the Lipschitz conditions \eqref{varphi_1_varphi_2_Lipschitz}, using \eqref{varphi_1_max} and \eqref{varphi_2_min} and applying Lemma \ref{lemma}, we obtain
        \begin{equation} \label{proof_2_conclusions}
            \begin{aligned}
                \partial_t \psi_1(t_\varepsilon^\delta, x_\varepsilon^\delta(\cdot))
                + \langle \nabla_{x(\cdot)} \psi_1(t_\varepsilon^\delta, x_\varepsilon^\delta(\cdot)),
                v \rangle
                & \geq - L_{\varphi_1, D}, \\
                \partial_\tau \psi_2(\tau_\varepsilon^\delta, y_\varepsilon^\delta(\cdot))
                + \langle \nabla_{y(\cdot)} \psi_2(\tau_\varepsilon^\delta, y_\varepsilon^\delta(\cdot)), v \rangle
                & \leq L_{\varphi_2, D}
            \end{aligned}
        \end{equation}
        for all $v \in B(1)$.
        Substituting $v = 0$ into the second inequality in \eqref{proof_2_conclusions} implies that $\partial_\tau \psi_2(\tau_\varepsilon^\delta, y_\varepsilon^\delta(\cdot)) \leq L_{\varphi_2, D}$.
        Hence, according to the first equality in \eqref{proof_derivatives_of_test_functionals} and since $\alpha > 0$, we have $\partial_t \psi_1(t_\varepsilon^\delta, x_\varepsilon^\delta(\cdot)) \leq L_{\varphi_2, D}$.
        Then, by the first inequality in \eqref{proof_2_conclusions},
        \begin{equation*}
            \langle \nabla_{x(\cdot)} \psi_1(t_\varepsilon^\delta, x_\varepsilon^\delta(\cdot)), v \rangle
            \geq - L_{\varphi_1, D} - L_{\varphi_2, D}
            \quad \forall v \in B(1),
        \end{equation*}
        which gives
        \begin{equation*}
            \|\nabla_{x(\cdot)} \psi_1(t_\varepsilon^\delta, x_\varepsilon^\delta(\cdot))\|
            \leq R,
            \quad R
            \coloneq L_{\varphi_1, D} + L_{\varphi_2, D} + 1.
        \end{equation*}

        Further reasoning is similar to Step 5 of the proof of Theorem \ref{theorem_local_comparison_1}.
        Take the number $L_{H, D, R}$ from assumption $(A.3)$ and fix $\varepsilon \in (0, \varepsilon_\ast]$ such that $L_{H, D, R} \sqrt{c} \sqrt{\varepsilon} \leq \alpha / 2$.
        Using assumption $(A.1)$, choose $\delta > 0$ such that inequality \eqref{H_difference} holds for all $t$, $\tau \in [0, T]$ with $|t - \tau| \leq \sqrt{c} \sqrt{\delta}$ and all $x(\cdot) \in D$, $s \in B(R)$.
        Suppose that $t_\varepsilon^\delta \geq \tau_\varepsilon^\delta$.
        Then, according to the first inequality in \eqref{proof_main_norms_varepsilon}, we have
        \begin{align*}
            & H \bigl( \tau_\varepsilon^\delta, x_\varepsilon^\delta(\cdot),
            \nabla_{x(\cdot)} \psi_1(t_\varepsilon^\delta, x_\varepsilon^\delta(\cdot)) \bigr)
            - H\bigl( \tau_\varepsilon^\delta, y_\varepsilon^\delta(\cdot), \nabla_{x(\cdot)} \psi_1(t_\varepsilon^\delta, x_\varepsilon^\delta(\cdot)) \bigr) \\
            & \quad \leq L_{H, D, R}
            \|x_\varepsilon^\delta(\cdot \wedge \tau_\varepsilon^\delta)
            - y_\varepsilon^\delta(\cdot \wedge \tau_\varepsilon^\delta)\|_\infty \\
            & \quad \leq L_{H, D, R} \sqrt{c} \sqrt{\varepsilon} \\
            & \quad \leq \frac{\alpha}{2}.
        \end{align*}
        Moreover, in view of the first inequality in \eqref{proof_basic_esimates_varepsilon_delta}, inequality \eqref{contradiction_1_proof_2} takes place.
        Thus, we conclude that inequality \eqref{contradiction_1} is valid (the case where $t_\varepsilon^\delta < \tau_\varepsilon^\delta$ can be handled similarly).

        From \eqref{contradiction_basic} and \eqref{contradiction_1}, it follows that $\alpha \leq 0$, which contradicts \eqref{proof_main_alpha}.
        The proof of Theorem \ref{theorem_local_comparison_2} is complete.
    \end{proof}

\end{document}